\documentclass[11pt, a4paper, twoside]{amsart}
\usepackage{amsthm, amsmath, amssymb, amscd, latexsym}
\usepackage{color}
\usepackage{graphicx}
\usepackage{soul}

\newtheorem{thm}{Theorem}[section]

\newtheorem{prop}[thm]{Proposition}

\newtheorem{lem}[thm]{Lemma}
\newtheorem{rem}[thm]{Remark}

\numberwithin{equation}{section}

\newcommand{\cZ}{\mathcal{Z}}

\title[Nonparametric Bayesian inference of diffusions]{Nonparametric Bayesian inference of discretely observed diffusions}
\author[J.-C. Croix, M. Dashti and I. Z. Kiss]{\sc Jean-Charles Croix, Masoumeh Dashti, and Istv\`an Zolt\`an Kiss}
\address{\noindent Department of Mathematics, University of Sussex, Brighton BN1 9QH, UK\\
 {j.croix@sussex.ac.uk},  {m.dashti@sussex.ac.uk}, {i.z.kiss@sussex.ac.uk}}
\date{\noindent \today}

\begin{document}
	
	\begin{abstract}	
\noindent We consider the problem of the Bayesian inference of drift and  diffusion coefficient functions in a stochastic differential equation, given discrete observations of a realisation of its solution. We give conditions for the well-posedness and stable approximations of the posterior measure. These conditions in particular allow for priors with unbounded support. Our proof relies on the explicit construction of transition probability densities using the parametrix method for general parabolic equations. We then study an application of these results in inferring the rates of Birth-and-Death processes.
	\end{abstract}
\maketitle
  
{\small    
   \noindent{\it Key words.}   inverse problems, Bayesian inference, diffusion processes, parametrix method, parabolic partial differential equations\\

   \noindent{\it AMS subject classifications.}   62G05, 62F15, 60J60, 65N21, 35K20\\
}    

	\section{Introduction}
	
	Stochastic processes are fundamental tools in the modelling of many real-world phenomena. In particular, stochastic differential equations (SDE) are often used in life sciences, engineering, economics and finance. Indeed, they provide a convenient way to represent dynamical systems including noise while they rely only on the specification of two functional coefficients of the drift and diffusion, and an initial condition. Such equations are usually represented as
	\begin{equation}
		dX_t=b(X_t,t)dt+\sqrt{a(X_t,t)}dW_t,\;t\in\mathbb{R}^+
		\label{Eq:SDEgeneral}
	\end{equation}
	where $X_0=x_0$ almost-surely, $(a,b)$ are respectively the diffusion and drift functions and $W=\left(W_t\right)_{t\in\mathbb{R}^+}$ is a standard Brownian motion on some filtered probability space $(\Omega,(\mathcal{F}_t)_{t\ge 0},\mathbb{P})$. Here we consider such processes restricted to a bounded interval, reduced without loss of generality to $I=[0,1]$. 
When it comes to applications, it is often the case that both the diffusion and the drift term are unknown, but we have access to some data generated by the process itself, mostly in the form of discrete observations, considered as a realisation from the random vector $(X_{s_1},\dots,X_{s_n})$ with discrete times $0\leq s_1<\dots<s_n\leq T$,  $n\in\mathbb{N}\setminus\lbrace 0\rbrace$, $T\in\mathbb{R}^+$. It is then natural to turn to statistical methods to infer these terms.
	
	Depending on both the nature of the SDE and the observables, a large catalogue of estimation methods are readily available, see for instance the reviews \cite{VanZanten2013,Fuchs2013} and references therein. However, the theoretical analysis of these problems, such as their well-posedness, stability or contraction rates for nonparametric approaches is far less investigated.
One main difficulty is that for low frequency data, one needs transition probability density functions as they underpin the likelihood. These, however, are not available in closed-forms and are only obtained as solutions to a particular type of partial differential equations. Indeed, if one has a vector of observations $y\in(0,1)^n$ of the process as described previously, the likelihood is
	\begin{equation}
		\mathcal{L}^{y,s}(a,b)=\prod_{i=1}^{n-1}p^{a,b}(x_{i+1},s_{i+1};x_i,s_i),
		\label{Eq:Likelihood}
	\end{equation}
	where we assumed that the solution $X$ to the SDE \eqref{Eq:SDEgeneral} admits $p^{a,b}(x,t;\xi,\tau)$ as transition probability density function. The variables $(x,t)$ and $(\xi,\tau)$ are respectively called forward and backward in time. Here $p^{a,b}$ is obtained as a fundamental solution (or Green function) of both Kolmogorov equations, the backward being
	\begin{gather}\label{Eq:KolmogorovBackward}
		-\frac{dp}{d\tau}=\frac{a(\xi,\tau)}{2}\frac{\partial^2 p}{\partial \xi^2}+b(\xi,\tau)\frac{\partial p}{\partial \xi},\\
		\lim_{\tau\to t}p(x,t;\xi,\tau)=\delta(\xi-x),
	\end{gather}
	and the forward
	\begin{gather*}
		\frac{dp}{dt}=\frac{1}{2}\frac{\partial^2(a(x,t)p)}{dx^2}-\frac{\partial (b(x,t)p)}{\partial x},\\
		\lim_{t\to \tau}p(x,t;\xi,\tau)=\delta(x-\xi).
	\end{gather*}
	In this work, the considered stochastic processes will satisfy either absorption or reflection on the boundaries. The former case corresponds to Dirichlet boundary conditions for both Kolmogorov equations. The latter will lead to Neumann and Robin boundary conditions respectively for the backward and forward Kolmogorov equations.
	
The literature on nonparametric inference methods for diffusions is vast \cite{Gobet2004,VanZanten2013,Pokern2013,Waaij2016,Nickl2017,Batz2018,Abraham2019,Gugushvili2020}, but only a few tackle the challenge of estimating both coefficients from discrete data. In \cite{Gobet2004}
the authors show that the problem of recovering both the drift and diffusion is ill-posed and they obtain minimax rates of convergence with a program using spectral theory under the low frequency regime (fixed time between data, increasing number of observations).
	More recently, the authors of \cite{Nickl2017} obtained minimax contraction rates of the posterior distribution for H\"older-Sobolev classes, providing theoretical guarantees of such approaches. Both these works consider reflected diffusions and priors with bounded support. 	
	
	In this paper, we study the situation where the amount of data is fixed and both drift and diffusion functions are unknown but H\"older continuous. We do not however specify a priori upper or lower bound for the coefficients (other than positivity for the diffusion coefficient). In this context, we provide a proof of the well-posedness of the Bayesian inference as well as a stability result with respect to appropriate approximations of the likelihood, both including unbounded priors.  In order to establish the necessary estimates and regularity properties of the likelihood (\ref{Eq:Likelihood}), our proof uses
the parametrix construction for fundamental solutions of parabolic equations \cite{Friedman1992} (see also \cite{Konakov2000, Konakov2017} for applications of this construction in the context of transition probability densities of Markov chains).   Here, we consider bounded domains with both Dirichlet (absorption) and Neumann/Robin (reflection) boundary conditions.  We follow the approach in \cite{Friedman1992} for construction of the fundamental solution, except that as the parametrix function, instead of the Gaussian density, we consider a series expansion obtained by the method of images \cite{Renshaw2015} that imposes the boundary conditions on the parametrix function (see Section \ref{s:altParametrix} below). 
These estimates then allow us to give sufficient conditions on the prior for the well-posedness of the posterior measure. These conditions, however, exclude a Gaussian prior. They require the tails of the prior to be thinner than those of the Gaussian measure. See Section \ref{subsec:SDEprior} for an example of a $\beta$-exponential prior with $\beta>3$ which satisfies the conditions on the prior.

We then apply our results to SDE approximations of a Birth-and-Death process where the up-jump rates are unknown.
In the SDE approximation of the process, these rates are translated to a function of the solution of the SDE that affects both the drift and the diffusion coefficients of the SDE. We construct an appropriate prior and formulate the posterior and also the variational functional whose minimisers give the maximum a posterior (MAP) estimators. Finally, in a numerical approximation we illustrate the performance of the conditional mean and the MAP estimators in estimating the truth.

	\section{Main results and assumptions}
	\label{sec:MainResults}

Let $(X_t)_{t\in\mathbb{R}^+}$ be a solution to  \eqref{Eq:SDEgeneral} which is either absorbed or reflected on the boundaries $\lbrace 0,1\rbrace$.
 What is referred to as data is a set of observations $y=(y_1,\dots,y_n)\in(0,1)^n$ at times $s:=(s_1,\dots,s_n)\in(0,T)^n$ with $n\geq 1$, $T\in(0,\infty)$ and $0\leq s_1<\dots<s_n\leq T$, taken as a realisation from the random vector $(X_{s_1},\dots,X_{s_n})$. 
 We suppose that $a$ and $b$ are H\"older continuous with the diffusion coefficient $a$ non-degenerate. In other words we suppose $(a,b)$ to lie in $\Lambda_\alpha$ given by
\begin{equation}\label{La}
 \Lambda_\alpha:=\Big\{(a,b)\in\mathcal{C}^{0,\alpha}\left([0,1]\times[0,T]\right)^2: m_a:=\min_{(x,t)\in[0,1]\times[0,T]}a(x,t)>0\Big\},
\end{equation} 
for some  $\alpha\in(0,1]$ and where the H\"older norm of a function $f\in\Lambda_\alpha$ is defined as
\begin{equation*}
			\lVert f\rVert_{0,\alpha}:=\lVert f\rVert_{L^\infty([0,1]\times [0,T])}
			+\sup_{(x,\xi,t,\tau)\in[0,1]^2\times[0,T]^2}\frac{\left\vert f(x,t)-f(\xi,\tau)\right\vert}{\vert x-\xi\vert^\alpha}.
\end{equation*}
For $(a,b)\in \Lambda_\alpha$ with some $\alpha\in(0,1]$, it is well-known that there exists a unique transition probability density function $p^{a,b}$ satisfying equation \eqref{Eq:KolmogorovBackward} (see e.g. \cite{Stroock1971}).  We will show that the likelihood from equation \eqref{Eq:Likelihood} leads to well-posed Bayesian inference problems (in the sense of \cite{Stuart2010}), when the prior measure satisfies a few technical conditions. This result is summarized in the two following Theorems (detailed proofs are given in Section \ref{sec:Proofs}).
In the following $d_{\rm H}$ denotes the Hellinger metric and for $\mu$ and $\mu'$ absolutely continuous with respect to a reference measure $\lambda$ on a space $\Lambda$, is the $L^2(\Lambda,\lambda)$ distance between $\sqrt{d\mu/d\lambda}$ and $\sqrt{d\mu'/d\lambda}$.
	\begin{thm}[Well-posedness]
		Let $T>0$, $n\geq 1$ and consider $y=(y_1,\dots,y_n)\in(0,1)^n$ to be the data at times $s=(s_1,\dots,s_n)\in(0,T)^n$ with $0<s_1\leq \dots\leq s_n\leq T$. Fix $\alpha\in (0,1]$ and suppose that  $\mu_0$ a probability measure with $\mu_0\left(\Lambda_\alpha\right)=1$ and there exists $C>0$ and $q>\frac{2}{\alpha}$ such that
		\begin{equation}
			\mathbb{E}^{\mu_0}\left[\exp\left(Cm_a^{1-q}\lVert (a,b)\rVert_{0,\alpha}^q\right)\right]<+\infty.
			\label{Eq:IntegrabilityCondition}
		\end{equation}
		Then there exists a unique posterior measure $\mu^{y,s}$ given by 
	\begin{equation*}
			\frac{d\mu^{y,s}}{d\mu_0}(a,b)=\frac{\mathcal{L}^{y,s}(a,b)}{\cZ(y,s)},
			\label{Eq:RadonNikodym}
		\end{equation*}
		where $\cZ(y,s):=\int_{\Lambda_\alpha}\mathcal{L}^{y,s}(a,b)d\mu_0\in\left(0,+\infty\right)$. Moreover, the posterior probability $\mu^{y,s}$ is continuous in $y$ with respect to Hellinger metric.
		\label{Thm:Main_thm}
	\end{thm}
	
	Our second result has important practical implications. It allows one, for instance, to obtain finite-dimensional approximations of the posterior measure given in Theorem \ref{Thm:Main_thm} using appropriate discretisations of the prior. See Section \ref{s:appex} for an example.
	
	\begin{thm}[Approximation]
		Under the assumptions of Theorem \ref{Thm:Main_thm}, suppose additionally that there exists a sequence $(a_k,b_k)\subset \Lambda_\alpha$ and a constant $C>0$ such that for all $k\in\mathbb{N}$
		\begin{equation*}
		\begin{split}
		0<m_a &\leq Cm_{a_k},\\
		\lVert (a_k,b_k)\rVert_{0,\alpha}&\leq C\lVert (a,b)\rVert_{0,\alpha},\\
		\lVert (a,b)-(a_k,b_k)\rVert_{0,\alpha}&\leq C\exp\left(C\lVert (a,b)\rVert_{0,\alpha}^q\right)\psi(k),
		\end{split}
		\end{equation*}
		with $\psi:k\in\mathbb{N}\to\mathbb{R}^+$ and $\lim_{k\to\infty}\psi(k)=0$. Then for any $k\in\mathbb{N}$ there exists a well defined probability measure $\mu_k^{y,s}$ given by
			\begin{equation*}
			\frac{d\mu_k^{y,s}}{d\mu_0}(a,b)=\frac{\mathcal{L}^{y,s}\left(a_k,b_k\right)}{\cZ_k(y,s)},
			\end{equation*}
			where $\cZ_k(y,s)=\int_{\Lambda_\alpha}\mathcal{L}^{y,s}\left(a_k,b_k\right)d\mu_0\in(0,\infty)$, satisfying
			\begin{equation*}
				d_{\rm H}(\mu^{y,s},\mu_k^{y,s})\leq C\sqrt{\psi(k)},
			\end{equation*}
			for all $k$ sufficiently large.
		\label{Thm:Approximation}
	\end{thm}

	\begin{rem}
		In Theorem \ref{Thm:Main_thm}, the prior measure $\mu_0$ on the coefficients must satisfy the integrability condition from equation \eqref{Eq:IntegrabilityCondition} with $q>\frac{2}{\alpha}$. This result may not be sharp and is inherent to our approach in carrying the bounds through the parametrix construction. In particular this forbids the use of Gaussian measures, even when $\alpha=1$. However, it is possible to specify priors with thinner tails than Gaussian measures, as it is presented in Section \ref{sec:Application} with Exponential priors.
	\end{rem}

	\begin{rem}
		The data $y$ is taken as an element from $(0,1)^n$ to avoid technicalities related to boundary behaviour of the stochastic process. However, it is straightforward to extend our result to the specific case of absorption. Indeed, say that $y_k=0$ with $k\in\lbrace 1,\dots,n\rbrace$ is the first observation of the process $(X_t)$ being absorbed at $0$. The transition probability density function $p^{a,b}$ in the likelihood from equation \eqref{Eq:Likelihood} must be replaced with the absorption probability:
		\begin{equation*}
			\mathbb{P}\left(X_{s_k}=0\vert X_{s_{k-1}}=y_{k-1}\right).
		\end{equation*}
		Obviously, the subsequent transitions can be ignored as they are non-informative ($y_{i}=0$ for all $k\leq i\leq n$).
	\end{rem}
	
	\begin{rem}
		We assume here that we observe data without measurement error. This assumption can be easily relaxed. Indeed, suppose for instance that the data is given by the following observation model for all $i\in\lbrace 1,\dots,n\rbrace$:
		\begin{equation*}
			\tilde{y}_i=X_{s_i}+\epsilon_i,\;i\in\lbrace 1,\dots,n\rbrace,
		\end{equation*}
		where $\epsilon=\left(\epsilon_1,\dots,\epsilon_n\right)$ is a vector of independent and identically distributed real random variables with probability density function $f$. The likelihood from equation \eqref{Eq:Likelihood} is then replaced by
		\begin{equation*}
			\mathcal{L}^{\tilde{y},s}\left(a,b\right)=\prod_{i=1}^{n-1}\tilde{p}^{a,b}(\tilde{y}_{i+1},s_{i+1};\tilde{y}_i,s_i),
		\end{equation*}
		with
		\begin{equation*}
			\tilde{p}^{a,b}(x,t;\xi,\tau)=\int_{-\infty}^{+\infty}\int_{-\infty}^{+\infty}p^{a,b}(x-w,t;\xi-z,\tau)f(w)f(z)dwdz.
		\end{equation*}
		Using Theorem \ref{Thm:ParabolicGreen}, this setting leads to analogous results to Theorems \ref{Thm:Main_thm} and \ref{Thm:Approximation}. 
	\end{rem}
	
	\section{Application to Birth-and-Death processes}
	\label{sec:Application}
	
	In this section, we will consider the problem of inferring BD processes rates from discrete observations. These dynamical systems are particularly useful in queuing theory and life sciences inspired models but their likelihood involves a large system of ordinary differential equations, quickly intractable as the dimension of the system increases. One common strategy consists in approximating rescaled BD processes with a continuous process \cite{Golightly2005}, taken as a particular Langevin SDE.
	
	\subsection{SDE approximation of density-dependent BD processes}
	\label{subsec:SDEBD}
	
	Consider a BD process $(Y_t)_{t\in\mathbb{R}^+}$ with state space $\lbrace 0,\dots,N\rbrace$ where $N\in\mathbb{N}^*$ and $T>0$ as described for example in \cite{Renshaw2015}. Its dynamics is parametrised by up and down jump rates $u_k$ and $d_k$ for all $k\in\lbrace 0,\dots,N\rbrace$ and an initial condition $Y_0$. In particular, the process will almost-surely remain in $\lbrace 0,\dots,N\rbrace$ under the sufficient conditions $u_N=d_0=0$ and $Y_0\in\lbrace 0,\dots,N\rbrace$. Suppose now that these jump rates have the so-called density dependent property \cite{Ethier2005}, meaning that there exists functions $U,D:[0,1]\to\mathbb{R}$ such that:
	\begin{equation*}
	\forall k\in\lbrace 0,\dots,N\rbrace,\;\frac{u_k}{N}=U\left(\frac{k}{N}\right)\text{ and }\frac{d_k}{N}=D\left(\frac{k}{N}\right),
	\end{equation*}
	and where we assume that $U$ and $D$ are non-negative on $[0,1]$. It is well-known (see for instance \cite{Kurtz1971}) that in the limit of large $N$, the rescaled BD process $(N^{-1}Y_t)_{t\in\mathbb{R}^+}$ can be approximated by a diffusion process, obtained as the solution of the following SDE:
	\begin{equation}
	dX_t=\left(U(X_t)-D(X_t)\right)dt+\sqrt{\frac{U(X_t)+D(X_t)}{N}}dW_t,
	\label{Eq:SDE_BD}
	\end{equation}
	with initial condition $X_0=N^{-1}Y_0\in[0,1]$ and for instance, reflective boundary conditions. If now we assume furthermore that $U+D$ is positive and both functions are sufficiently regular, then equation \eqref{Eq:SDE_BD} can be considered within the Bayesian methodology developed in Section \ref{sec:MainResults}. In particular, one can consider $b:=U-D$ and $a:=\frac{U+D}{N}$ and then obtain estimators for $U,D$ as $U=\frac{Na+b}{2}$ and $D=\frac{Na-b}{2}$. However, we will follow a different route, assuming here (for simplicity) that $D$ is known. In applications such as models for the transmission of diseases, the infection process is more difficult to characterise (it involves the particular features of the disease as well as the contact pattern between individuals) while the recovery in most cases is much easier to observe. In particular, Susceptible-Infected-Susceptible (SIS) models include linear down jump rates $d_k=\gamma k$ with $\gamma>0$ the recovery rate. This leads to $D(x)=\gamma x$ \cite{Anderson1992,Kiss2017}.

	\subsection{Exponential prior measure and Bayes Theorem}
	\label{subsec:SDEprior}
	
	In this application, we are interested in estimating the up jump coefficient $U$ when $D$ is known. Specifying a distribution on $U$ thus implies a prior measure $\mu_0$ on $\Lambda_\alpha$ following the previous discussion. In particular, we will build a random function $U$ such that $(a,b)$ satisfy the following properties:
	\begin{itemize}
		\item H\"older continuity and uniform ellipticity: $\exists \alpha\in(0,1],\;(a,b)\in \Lambda_\alpha$ $\mu_0$-a.s.,
		\item Integrability of the likelihood with respect to $\mu_0$:
		\begin{equation}
		\exists q>\frac{2}{\alpha},\;\mathbb{E}^{\mu_0}\left[\exp\left(Cm_a^{1-q}\lVert (a,b)\rVert_{0,\alpha}^q\right)\right]<\infty,
		\end{equation}
	\end{itemize}
	and thus Theorem \ref{Thm:Main_thm} applies as it is and provides a well-defined posterior on $(a,b)$ (which in fact, can be seen as a posterior distribution for $U$). The random function $U$ will be built as the image of a random series through an appropriate deterministic map (the map is used to impose positivity of $a$ and the condition $U(1)=0$). Let $(f_k)_{k\in\mathbb{N}}$ be the Fourier basis of $L^2(0,1)$, that is for all $k\in\mathbb{N}$ and for all $x\in[0,1]$,
	\begin{equation*}
		\begin{split}
		f_{2k}(x)&=\sqrt{2}\cos(2k\pi x),\\
		f_{2k+1}(x)&=\sqrt{2}\sin(2k\pi x).
		\end{split}
	\end{equation*}
	Consider a sequence of independent and identically distributed random variables $(\eta_k)_{k\in\mathbb{N}}$ with density function ($\beta$-exponential random variables where $\beta >2$):
	\begin{equation*}
		\pi(x)\propto\exp\left(-\vert x\vert^\beta\right),
	\end{equation*}
	and a sequence of decreasing positive numbers $(\gamma_k)_{k\in\mathbb{N}}$ such that for some constants $c_1,c_2>0$,
	\begin{equation*}
	c_1 \,k^{-\theta}\le\gamma_k\le c_2\, k^{-\theta},
	\end{equation*}
	with $\theta>0$. Define a random series $f$ in the probability space $(\mathbb{R}^\infty,\mathbb{B}(\mathbb{R}^\infty),\mathbb{P}_0)$ with $\mathbb{B}(\mathbb{R}^\infty)$ the product $\sigma$-algebra on $\mathbb{R}^\infty$, as follows:
	\begin{equation}
	f=\sum_{k\geq 0}\gamma_k\eta_kf_k.
	\label{Eq:RandomSeries}
	\end{equation}
	Furtheremore, define the following Hilbert scale of Sobolev spaces:
	\begin{equation*}
	H^l(0,1)=\left\lbrace \sum_{k\geq 0}u_kf_k,\;\sum_{k\geq 0}k^{2l}u_k^2<\infty\right\rbrace,
	\end{equation*}
	with norms $\lVert\sum_{k\geq 0}u_kf_k\rVert_{H^l}=\left(\sum_{k\geq 0}k^{2l}u_k^2\right)^\frac{1}{2}$ and where for instance $H^0=L^2$ and for integer values of $l$ one recovers the usual $H^l$ Sobolev spaces. Immediate properties of this random series are given in the following proposition, with in particular an exponential moment of order $\beta$.
	\begin{prop}
		Let $f$ be defined as in equation \eqref{Eq:RandomSeries} and suppose that $\beta>2$, $\alpha\in(0,1]$, $l> \alpha+\frac{1}{2}$ and $\theta>2l+1-2\beta^{-1}$ then there exists $C>0$ such that:
		\begin{equation*}
			\mathbb{E}\left[\exp\left(C\lVert f\rVert_{H^l}^\beta\right)\right]<+\infty.
		\end{equation*}
		In particular, one has for all $\alpha\in\left(0,1\right]$:
		\begin{itemize}
			\item[i.] $f\in\mathcal{C}^{0,\alpha}$ $\mathbb{P}_0$-almost-surely,
			\item[ii.] there exists $C>0$ such that $\mathbb{E}\left[\exp\left(C\lVert f\rVert_{0,\alpha}^\beta\right)\right]<+\infty$.
		\end{itemize}
	\label{Prop:RandomSeries}
	\end{prop}	
\begin{rem}
	We note that for the Fourier series considered in equation \eqref{Eq:RandomSeries} we have
	\begin{align*}
	|f_k(x)|&\le 1,\\
	|f_k(x)-f_k(y)|&\le Ck|x-y|,
	\end{align*}
	and hence by Theorem 2.8 of \cite{Dashti2015}, part (i) of the above proposition follows from the less stringent condition of $\theta>\alpha+\frac{1}{\beta}$. 
	It is for the result of part (ii)  that we choose $\theta$ sufficiently large so that a straightforward adaptation (through Sobolev embedding of $H^l$ in the H\"older space of our interest) of the method used in \cite{Lassas2009}, works here. 
\end{rem}

We now consider $g:\mathbb{R}\to \mathbb{R}$ and $h:[0,1]\to \mathbb{R}$ as follows
	\begin{align*}
	&g(x):=
	\left\lbrace
	\begin{array}{l}
	x,\;x\geq 1,\\ 
	\frac{1}{2-x},\;x\leq 0,\\
	\tilde{g}(x),\;0<x<1,
	\end{array} 
	\right.\\
	&h(x):=1-\exp(x-1),
	\end{align*}
	with $\tilde{g}$ is chosen such that $g$ is positive, increasing, infinitely differentiable. Note in particular that $g$ has an inverse. Then the coefficient $U$ will be modelled as
	\begin{equation}
		U(x)=g(f(x))h(x),\;x\in[0,1].
		\label{Eq:U}
	\end{equation}
We hance have by this construction that $U\sim \nu_0$ with probability measure $\nu_0$ satisfying $\nu_0(C_0^{0,\alpha})=1$ where 
$$
C_0^{0,\alpha}:=\{U\in C^{0,\alpha}: U(x)>0 \mbox{ for }x\in[0,1), U(1)=0 \}.
$$ 
 Now, using $D(x)=\gamma x$ with $\gamma>0$ (known), both drift and diffusion coefficients are fully specified:
	\begin{equation}
	\begin{split}
	a(x)&=\frac{U(x)+D(x)}{N},\\
	b(x)&=U(x)-D(x).
	\end{split}
	\label{Eq:abU}
	\end{equation}
	In the following we let
	$$
	\mathcal{L}^{y,s}_{BD}(U):=\mathcal{L}^{y,s}((U+D)/N,U-D).
	$$
	The next proposition is an analogue of Theorem \ref{Thm:Main_thm} in this particular context.
	
	\begin{prop}
		Let $\alpha\in(0,1]$, $l>\alpha+\frac{1}{2}$, $q>\frac{2}{\alpha}$, $\beta>2q-1$ and $\theta>2l+1-2\beta^{-1}$. Suppose that $U\sim\nu_0$ is defined as in equation \eqref{Eq:U} and $D(x)=\gamma x$ with $\gamma >0$. 
Let $y$ and $s$ be as in Theorem \ref{Thm:Main_thm}. Then there exists a unique posterior measure $\nu^{y,s}$ over $C_0^{0,\alpha}$ given by
		\begin{equation}
			\frac{d\nu^{y,s}}{d\nu_0}(U)=\frac{\mathcal{L}^{y,s}_{BD}(U)}{\cZ^\nu(y,s)},
			\label{Eq:RadonNikodymU}
		\end{equation}
where $\cZ^{\nu}(y,s):=\int_{C^{0,\alpha}_0}\mathcal{L}^{y,s}_{BD}(U)d\nu_0\in\left(0,+\infty\right)$.
Moreover, the measure $\nu^{y,s}$ is continuous  in $y$ in Hellinger's metric. 
		\label{Prop:WellposedU}
	\end{prop}

The following proposition shows that finding the modes of the posterior measure $\nu^{y,s}$ can be characterised as a variational problem. 
It follows by the same arguments as those of \cite{Helin2015} for the case of differentiable Besov prior measures (note that we are using a Fourier basis here).
See Section \ref{s:BDproofs} for a short proof explaining the connections.

\begin{prop}\label{p:MAPs}
Under the assumptions of Proposition \ref{Prop:WellposedU}, the MAP estimators of posterior measure $\nu^{y,s}$ coincide with the minimisers of the functional
\begin{equation}\label{e:MAPs}
 -\log \mathcal{L}^{y,s}_{BD}(U)+\|g^{-1}(U/h)\|_{E}^\beta,
\end{equation}
over 
$$
E:=\Big\{u=\sum_{k\ge 1}u_kf_k\in L^2:\sum_{k\ge 0}\gamma_k^{-\beta}|u_k|^\beta<\infty\Big\},
$$
with $\|u\|_E^\beta:=\sum_{k\ge 0}\gamma_k^{-\beta}|u_k|^\beta$.
\end{prop}

	\subsection{Approximation}\label{s:appex}
	
	In the previous section, we have constructed a prior distribution on $\Lambda_\alpha$ by first randomising $U$ using a deterministic transform of the random series $f$. Here, we consider the approximation which consists in keeping a finite number of components in this random series. We define for this purpose the following approximations:
	\begin{equation}
		\forall k\in\mathbb{N},\;U_k=g\left(P_kf\right)h,
		\label{Eq:ApproxU}
	\end{equation}
	where $P_kf=\sum_{i=0}^{k}\gamma_i\eta_if_i$. This naturally translates into a sequence of approximations for the drift and diffusion coefficients for all $k\in\mathbb{N}$:
	\begin{equation*}
		\begin{split}
		a_k&=\frac{U_k+D}{N},\\
		b_k&=U_k-D.
		\end{split}
	\end{equation*}
	The next proposition shows how this approximation affects the posterior measure $\nu^{y,s}$.
	\begin{prop}
	Let the assumptions of Proposition \ref{Prop:WellposedU} hold, and suppose that $\{U_k\}\subset C^{0,\alpha}_0$ be defined as in equation (\ref{Eq:ApproxU}). Then for any $k\in\mathbb{N}$ there exists a well-defined probability measure $\nu^{y,s}_k$,
\begin{equation*}
			\frac{d\nu^{y,s}_k}{d\nu_0}(U)=\frac{\mathcal{L}^{y,s}_{BD}(U_k)}{\cZ^\nu(y,s)},
		\end{equation*}
with $\cZ_k^{\nu}(y,s):=\int_{C^{0,\alpha}_0}\mathcal{L}^{y,s}_{BD}(U_k)d\nu_0\in\left(0,+\infty\right)$, and we have
		\begin{equation*}
			d_{\rm H}(\nu^{y,s},\nu_k^{y,s})\leq Ck^{\alpha+\frac{1}{2}-l}
		\end{equation*}
		for $k$ sufficiently large.
		\label{Prop:ApproxU}
	\end{prop}
	
	\subsection{Numerical example}
	
	Finally, we provide a numerical illustration of the previous methodology. Here, we consider $n=100$ observations of a single trajectory, that is $y\in(0,1)^n$ is simulated using a simple Euler-scheme with regular time-steps $0.1=s_1\leq \dots\leq s_n=T=10$. All data locations (without time ordering) can be seen as blue dots in the bottom plot of figure \ref{fig:Posterior}. The {true} coefficients in this simulation are $\gamma=\frac{1}{2}$, $N=100$ and $U(x)=1-x^2$ with $X_0=0.1$ at time $s_0=0$, the prior is truncated to $k=100$ components, we chose $\alpha=1$ which leads to $\beta=3.01$, $\theta=4$ and the likelihood is numerically computed using a second order Chang-Cooper finite difference scheme \cite{Chang1970,Mohammadi2015}. Samples from the prior distribution are shown in figure \ref{fig:Posterior} (top) by direct sampling of $(\eta_0,\dots,\eta_{99})$ multiple times. The posterior samples are obtained using a whitened version of the pCN algorithm, see \cite{Chen2018}. The conditional mean (CM) is computed out of these samples. The MAP estimator is obtained as the minimizer, among the posterior sample of the following generalized Onsagher-Mashlup functional:
		\begin{equation*}
		\Psi(\eta_0,\dots,\eta_{99})=-\log\mathcal{L}^{y,s}_{BD}\Big(g\big(\sum_{i=0}^{99}\gamma_i\eta_i\big)h\Big)+\sum_{i=0} ^{99}\vert\eta_i\vert^\beta.
	\end{equation*}
	Both estimators are shown in figure \ref{fig:Posterior}, bottom plot. Here, one can observe that the posterior distribution shifts and contract around the true value (blue curve in figure \ref{fig:Posterior}).
	
	\begin{figure}[!h]
		\centering
		\includegraphics[width=4in,height=4in]{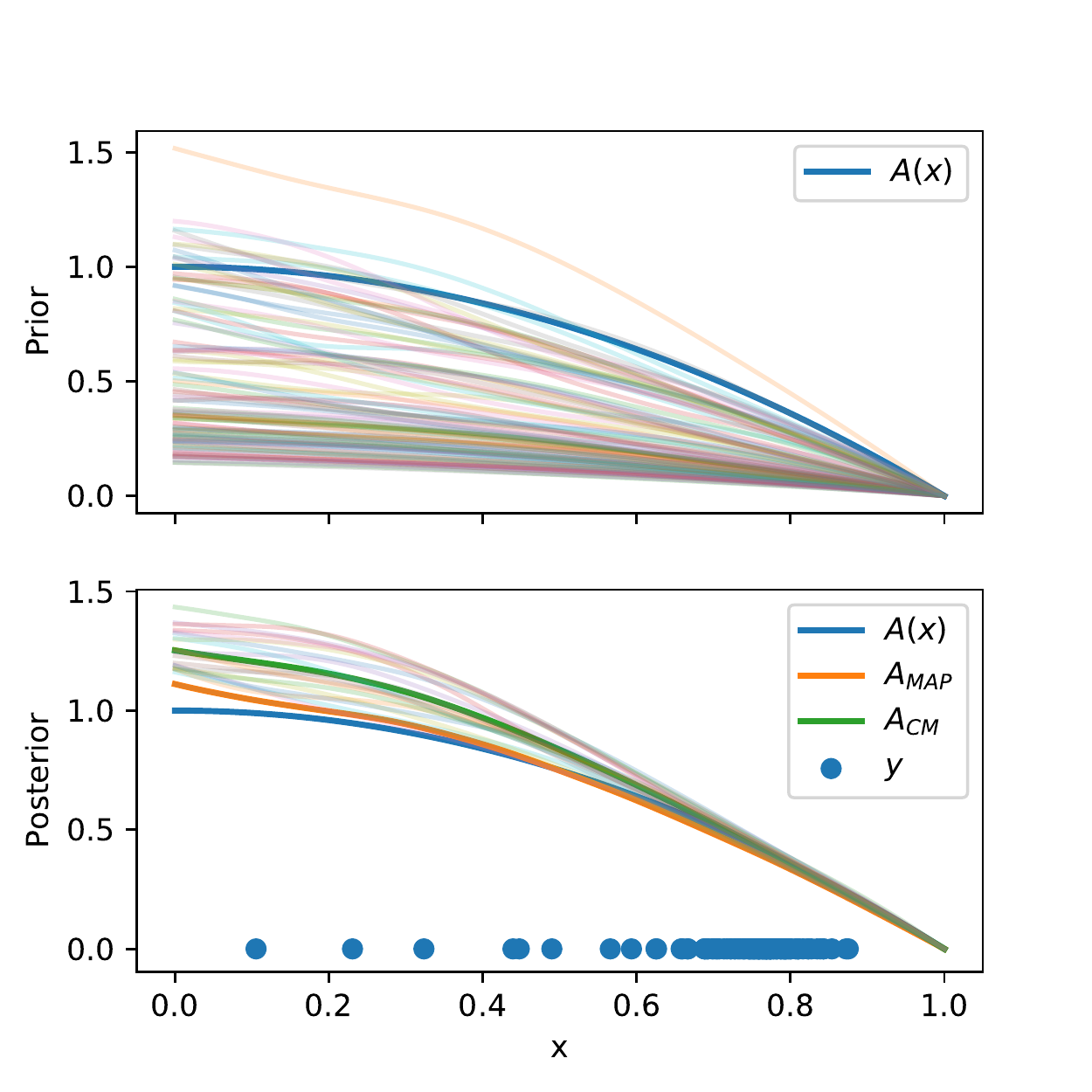}
		\label{fig:Posterior}
		\caption{Samples from the prior (top) and posterior (bottom) distributions on $U$ as well as MAP and CM estimators. Thick blue curves represent the true parameter $U(x)=1-x^2$. Blue dots represent data locations.}
	\end{figure}
	
	\section{Proofs}
	\label{sec:Proofs}
	
	In order to prove Theorem \ref{Thm:Main_thm} \& \ref{Thm:Approximation}, our strategy consists in establishing general results for Green functions associated to one-dimensional uniformly parabolic equations (Subsection \ref{subsec:Green}) and then apply them to the special case of Kolmogorov backward equation for SDEs (Subsection \ref{subsec:SDE}). This is motivated by observing that the backward Kolmogorov equation \eqref{Eq:KolmogorovBackward} is parabolic with the inversion of time $v=t-\tau$:
	\begin{gather*}
		\frac{\partial p}{\partial v}=\frac{a(\xi,t-v)}{2}\frac{\partial^2 p}{\partial \xi^2}+b(\xi,t-v)\frac{\partial p}{\partial \xi},\\
		\lim_{s\to 0}p(x,t;\xi,t-v)=\delta(\xi-x).
	\end{gather*}
	In the rest of this document, every reference to a constant $C$ may change from a line to the next.
	
	\subsection{Green functions of parabolic equations}
	\label{subsec:Green}
	
	In this section, we will present elementary results from the theory of non-degenerate parabolic partial differential equations in non-divergence form. We consider the following model equation:
	\begin{equation}
	\begin{split}
	\frac{\partial p}{\partial t}=a(x,t)\frac{\partial^2 p}{\partial x^2}+b(x,t)\frac{\partial p}{\partial x},
	\end{split}
	\label{Eq:Parabolic}
	\end{equation}
	with either Dirichlet or Neumann boundary conditions on the one-dimensional bounded domain $I=(0,1)$, and $(a,b)\in \Lambda_\alpha$ for $\alpha\in(0,1]$. In the rest of this paper and without loss of generality, we will always assume that $\lVert a\rVert_{0,\alpha}$, $\lVert b\rVert_{0,\alpha}$ are larger than one, $m_a\leq 1$ and we will use the notations $\lVert (a,b)\rVert_{0,\alpha}:=\lVert a\rVert_{0,\alpha}+\lVert b\rVert_{0,\alpha}$. Indeed, we can always replace $m_a$ in the following analysis by $1\wedge m_a$ or $\lVert a\rVert_{0,\alpha}$, $\lVert b\rVert_{0,\alpha}$ by $1\vee\lVert a\rVert_{0,\alpha}$, $1\vee\lVert a\rVert_{b,\alpha}$ respectively as we are after upper bounds. Now, for $0<\alpha\leq 1,\;(a,b)\in \Lambda_\alpha$, a well-known approach to construct fundamental solutions is the parametrix method, as presented in e.g. \cite{Friedman1992}. In the current context of Bayesian inference, it is also needed to study both the dependence (smoothness and bounds) of such fundamental solutions (or Green function) with respect to its coefficients. This is the objective of Theorem \ref{Thm:ParabolicGreen}, where the classical Gaussian upper bound (Nash-Aronson) is given with explicit dependence in $(a,b)$ together with a local Lipschitz continuity result.
	\begin{thm}
		Let $T>0$, $\alpha\in(0,1]$ and $(a,b)\in \Lambda_\alpha$. Then equation \eqref{Eq:Parabolic} with Dirichlet or Neumann boundary conditions has a unique Green function $G^{a,b}$ satisfying the following properties for $0\leq \tau<t\leq T$:
		\begin{itemize}
			\item (Continuity): $G^{a,b}$ is continuous in $(x,t;\xi,\tau;a,b)$,
			\item (Gaussian-type upper bound): for all $q>\frac{2}{\alpha}$ there exists $C>0$,
			\begin{equation*}
			\forall (x,\xi)\in[0,1]^2,\;\left\vert G^{a,b}(x,t;\xi,\tau)\right\vert\leq \frac{C}{\sqrt{t-\tau}}\exp\left(Cm_a^{1-q}\lVert (a,b)\rVert_{0,\alpha}^q\right),
			\end{equation*}
			\item (Local Lipschitz continuity in the coefficients): let $(a,b),(a',b')\in \Lambda_\alpha$, $q>\frac{2}{\alpha}$ then for all $(x,\xi)\in[0,1]^2$,
			\begin{equation*}
			\begin{split}
			&\left\vert G^{a,b}(x,t;\xi,\tau)-G^{a',b'}(x,t;\xi,\tau)\right\vert\\
			&\leq \frac{C}{\sqrt{t-\tau}}\exp\left(C\left(m_a\wedge m_{a'}\right)^{1-q}\left(\lVert (a,b)\rVert_{0,\alpha}\vee\lVert (a',b')\rVert_{0,\alpha}\right)^q\right)\\
			&\times\lVert (a,b)-(a',b')\rVert_{0,\alpha}.
			\end{split}
			\end{equation*}
		\end{itemize}
		\label{Thm:ParabolicGreen}
	\end{thm}
	
	The rest of this section will be dedicated to the proof of this Theorem with particular emphasis on the new results which are the two last properties.
	
	\subsubsection{Parametrix functions}\label{s:altParametrix}
	
	In the classical construction from \cite{Friedman1992}, the parametrix function is taken to be the fundamental solution of the Heat equation with constant diffusion $2a(\xi,\tau)$ (the variables $(\xi,\tau)$ being fixed). Here, we use an alternative parametrix functions, already including boundary conditions (Dirichlet or Neumann). Indeed, the Heat equation with a constant diffusion coefficient can be solved in one dimension by the method of images as described in e.g. \cite{Renshaw2015}. Let $Z^a_D$ and $Z^a_N$ be respectively the parametrix (Heat fundamental solution) with Dirichlet and Neumann boundary conditions, then they can be represented as follows where $T>0$, $(x,\xi)\in[0,1]^2$, $0\leq \tau<t\leq T$ and $(a,0)\in \Lambda_\alpha$:
	\begin{equation}
	\begin{split}
	&Z^a_D(x,t;\xi,\tau)=\\
	&\frac{1}{\sqrt{4\pi a(\xi,\tau)(t-\tau)}}\sum_{n=-\infty}^{+\infty}
	\left(\exp\left(-\frac{(x-y_n)^2}{4a(\xi,\tau)(t-\tau)}\right)-\exp\left(-\frac{(x-z_n)^2}{4a(\xi,\tau)(t-\tau)}\right)\right),
	\end{split}
	\label{Eq:Parametrix_Dirichlet}
	\end{equation}
	\begin{equation}
	\begin{split}
	&Z^a_N(x,t;\xi,\tau)=\\
	&\frac{1}{\sqrt{4\pi a(\xi,\tau)(t-\tau)}}\sum_{n=-\infty}^{+\infty}
	\left(\exp\left(-\frac{(x-y_n)^2}{4a(\xi,\tau)(t-\tau)}\right)+\exp\left(-\frac{(x-z_n)^2}{4a(\xi,\tau)(t-\tau)}\right)\right),
	\end{split}
	\label{Eq:Parametrix_Neumann}
	\end{equation}
	where for all $n\in\mathbb{Z}$, $y_n=\xi+2n$ and $z_n=-\xi+2n$. The main reason to consider those functions in the parametrix construction instead of a Gaussian density lies in the fact that they will immediately provide Green functions satisfying boundary conditions. This avoids non-trivial modifications depending on the type of boundary conditions \cite{Friedman1992}. We now start by providing the necessary properties of parametrix functions to be compatible with the classical construction to come.
	
	\begin{prop}
		Let $Z^a$ be either the function defined in equation \eqref{Eq:Parametrix_Dirichlet} or \eqref{Eq:Parametrix_Neumann} with $\alpha\in(0,1]$, $(a,0)\in \Lambda_\alpha$ and $T>0$. Then for $k\in\lbrace 0,1,2\rbrace$, $(x,\xi)\in [0,1]^2$ and $0\leq \tau<t\leq T$ we have
		\begin{equation*}
		\left\vert\frac{\partial^kZ^a}{\partial x^k}(x,t;\xi,\tau)\right\vert\leq \frac{C}{a(\xi,\tau)^\mu(t-\tau)^\mu\vert x-\xi\vert^{k+1-2\mu}}\exp\left(-\frac{(1-\epsilon)(x-\xi)^2}{4a(\xi,\tau)(t-\tau)}\right),
		\end{equation*}
		with $0<\epsilon<1$ and $\mu\in[0,\frac{k+1}{2}]$. Furthermore, when $(a',0)\in \Lambda_\alpha$ one has
		\begin{equation*}
		\begin{split}
		&\left\vert\frac{\partial^kZ^a}{\partial x^k}(x,t;\xi,\tau)-\frac{\partial^kZ^{a'}}{\partial x^k}(x,t;\xi,\tau)\right\vert\\
		&\leq \frac{C}{\left(m_a\wedge m_{a'}\right)^{\max\left(1,\frac{2k+1}{2}\right)+\mu}(t-\tau)^\mu\vert x-\xi\vert^{k-2\mu}}\\
		&\times\exp\left(-\frac{(1-\epsilon)(x-\xi)^2}{4\left(\lVert a\rVert_{0,\alpha}\vee\lVert a'\rVert_{0,\alpha}\right)(t-\tau)}\right)\lVert a-a'\rVert_{0,\alpha},
		\end{split}
		\end{equation*}
		with $\mu$ same as above.
		\label{Prop:Parametrix}
	\end{prop}
	
	\begin{proof}
		We start by observing that for all $(x,\xi)\in [0,1]^2$, $(t,\tau)\in[0,T]^2$, $T>0$, $0\leq \tau<t\leq T$, the function $y\to\exp\left(-\frac{(x-\xi-2y)^2}{4a(\xi,\tau)(t-\tau)}\right)$ is decreasing on $\left[\frac{x-\xi}{2},+\infty\right)$ and using a series/integral comparison, one obtains:
		\begin{equation*}
		\sum_{n\geq 2}
		\exp\left(-\frac{(x-y_n)^2}{4a(\xi,\tau)(t-\tau)}\right)\leq C\exp\left(-\frac{(x-\xi)^2}{4a(\xi,\tau)(t-\tau)}\right).
		\end{equation*}
		Indeed, taking $z=\frac{2y-x+\xi}{\sqrt{2a(\xi,\tau)(t-\tau)}}$, one has:
		\begin{equation*}
		\begin{split}
		\sum_{n\geq 2}\exp\left(-\frac{(x-y_n)^2}{4a(\xi,\tau)(t-\tau)}\right)&\leq\int_{1}^{+\infty}\exp\left(-\frac{(x-\xi-2y)^2}{4a(\xi,\tau)(t-\tau)}\right)dy,\\
		&\leq\int_{\frac{2-x+\xi}{\sqrt{2a(\xi,\tau)(t-\tau)}}}^{+\infty}\frac{2}{\sqrt{2a(\xi,\tau)(t-\tau)}}\exp\left(-\frac{z^2}{2}\right)dz,\\
		&\leq\frac{2\sqrt{2\pi}}{\sqrt{2a(\xi,\tau)(t-\tau)}}\left(1-\Phi\left(\frac{2-x+\xi}{\sqrt{2a(\xi,\tau)(t-\tau)}}\right)\right),\\
		&\leq \frac{2}{2-x+\xi}\exp\left(-\frac{\left(2-x+\xi\right)^2}{4a(\xi,\tau)(t-\tau)}\right),\\
		&\leq C\exp\left(-\frac{\left(x-\xi\right)^2}{4a(\xi,\tau)(t-\tau)}\right),
		\end{split}
		\end{equation*}
		where we used the following bound for $1-\Phi(x)$:
		\begin{equation*}
		\begin{split}
		\forall x>0,\;1-\Phi(x)
		\leq\int_{x}^{+\infty}\frac{y}{x\sqrt{2\pi}}\exp\left(-\frac{y^2}{2}\right)dy
		=\frac{1}{x\sqrt{2\pi}}\exp\left(-\frac{x^2}{2}\right).
		\end{split}
		\end{equation*}
		A similar calculation leads to:
		\begin{equation*}
		\sum_{n\leq -2}\exp\left(-\frac{(x-y_n)^2}{4a(\xi,\tau)(t-\tau)}\right)\leq C\exp\left(-\frac{\left(x-\xi\right)^2}{4a(\xi,\tau)(t-\tau)}\right),
		\end{equation*}
		and it finally comes, considering the remaining terms $n\in\lbrace -1,0,1\rbrace$ that:
		\begin{equation*}
		\sum_{n=-\infty}^{+\infty}\exp\left(-\frac{(x-y_n)^2}{4a(\xi,\tau)(t-\tau)}\right)\leq C\exp\left(-\frac{\left(x-\xi\right)^2}{4a(\xi,\tau)(t-\tau)}\right).
		\end{equation*}
		A similar method gives:
		\begin{equation*}
		\sum_{n=-\infty}^{+\infty}\exp\left(-\frac{(x-z_n)^2}{4a(\xi,\tau)(t-\tau)}\right)\leq C\exp\left(-\frac{\left(x-\xi\right)^2}{4a(\xi,\tau)(t-\tau)}\right),
		\end{equation*}
		which leads to the announced upper bound for both $Z^a_N$ and $Z^a_D$ by the triangle inequality. Consider now the first derivative, bounded as follows:
		\begin{equation*}
		\begin{split}
		&\left\vert\frac{\partial Z^a}{\partial x}(x,t;\xi,\tau)\right\vert\\
		&\leq\frac{C}{a(\xi,\tau)^\frac{3}{2}(t-\tau)^\frac{3}{2}}\\
		&\times\sum_{n=-\infty}^{+\infty}
		\left((x-y_n)\exp\left(-\frac{(x-y_n)^2}{4a(\xi,\tau)(t-\tau)}\right)+(x-z_n)\exp\left(-\frac{(x-z_n)^2}{4a(\xi,\tau)(t-\tau)}\right)\right).
		\end{split}
		\end{equation*}
		As one has for all $n\in\mathbb{Z}$:
		\begin{equation*}
		\frac{(x-y_n)}{\sqrt{a(\xi,\tau)(t-\tau)}}\exp\left(-\frac{(x-y_n)^2}{4a(\xi,\tau)(t-\tau)}\right)\leq C\exp\left(-\frac{(1-\epsilon)(x-y_n)^2}{4a(\xi,\tau)(t-\tau)}\right),
		\end{equation*}
		with $C$ independent of $n$ and $\epsilon\in(0,1)$, it follows that:
		\begin{equation*}
		\begin{split}
		&\left\vert\frac{\partial Z^a}{\partial x}(x,t;\xi,\tau)\right\vert\\
		&\leq \frac{C}{a(\xi,\tau)(t-\tau)}\sum_{n=-\infty}^{+\infty}
		\frac{\vert x-y_n\vert}{\sqrt{a(\xi,\tau)(t-\tau)}}\exp\left(-\frac{(x-y_n)^2}{4a(\xi,\tau)(t-\tau)}\right)\\
		&+\frac{C}{a(\xi,\tau)(t-\tau)}\sum_{n=-\infty}^{+\infty}\frac{\vert x-z_n\vert}{\sqrt{a(\xi,\tau)(t-\tau)}}\exp\left(-\frac{(x-z_n)^2}{4a(\xi,\tau)(t-\tau)}\right),\\
		&\leq \frac{C}{a(\xi,\tau)(t-\tau)}\sum_{n=-\infty}^{+\infty}
		\exp\left(-\frac{(1-\epsilon)(x-y_n)^2}{4a(\xi,\tau)(t-\tau)}\right)
		+\exp\left(-\frac{(1-\epsilon)(x-z_n)^2}{4a(\xi,\tau)(t-\tau)}\right),\\
		&\leq \frac{C}{a(\xi,\tau)(t-\tau)}\exp\left(-\frac{(1-\epsilon)\left(x-\xi\right)^2}{4a(\xi,\tau)(t-\tau)}\right),\\
		&\leq \frac{C}{a(\xi,\tau)^\mu(t-\tau)^\mu\vert x-\xi\vert^{2-2\mu}}
		\exp\left(-\frac{(1-\epsilon)\left(x-\xi\right)^2}{4a(\xi,\tau)(t-\tau)}\right).
		\end{split}
		\end{equation*}
		A similar computation provides the bound for $\frac{\partial^2 Z^a}{\partial x^2}$. Regarding the local Lipschitz continuity in the coefficient $a$, we first consider the following relation, given for every $n\in\mathbb{Z}$ (by Lemma \ref{Lem:Taylor} in the appendix):
		\begin{equation*}
		\begin{split}
		&\left\vert\exp\left(-\frac{(x-y_n)^2}{4a(\xi,\tau)(t-\tau)}\right)-
		\exp\left(-\frac{(x-y_n)^2}{4a'(\xi,\tau)(t-\tau)}\right)\right\vert\\
		&\leq
		C\frac{\vert a(\xi,\tau)-a'(\xi,\tau)\vert}{\left(a(\xi,\tau)\wedge a'(\xi,\tau)\right)}\exp\left(-\frac{(1-\epsilon)(x-y_n)^2}{4\left(a(\xi,\tau)\vee a'(\xi,\tau)\right)(t-\tau)}\right),
		\end{split}
		\end{equation*}
		which immediately gives:
		\begin{equation*}
		\begin{split}
		&\sum_{n=-\infty}^{+\infty}\left\vert
		\exp\left(-\frac{(x-y_n)^2}{4a(\xi,\tau)(t-\tau)}\right)-
		\exp\left(-\frac{(x-y_n)^2}{4a'(\xi,\tau)(t-\tau)}\right)\right\vert\\
		&\leq C\frac{\vert a(\xi,\tau)-a'(\xi,\tau)\vert}{\left(a(\xi,\tau)\wedge a'(\xi,\tau)\right)}
		\sum_{n=-\infty}^{+\infty}\exp\left(-\frac{(1-\epsilon)(x-y_n)^2}{4\left(a(\xi,\tau)\vee a'(\xi,\tau)\right)(t-\tau)}\right)\\
		&\leq C\frac{\vert a(\xi,\tau)-a'(\xi,\tau)\vert}{\left(a(\xi,\tau)\wedge a'(\xi,\tau)\right)}\exp\left(-\frac{(1-\epsilon)\left(x-\xi\right)^2}{4\left(a(\xi,\tau)\vee a'(\xi,\tau)\right)(t-\tau)}\right).
		\end{split}
		\end{equation*}
		The same applies to the series involving $z_n$. Now consider the following calculation:
		\begin{equation*}
		\begin{split}
		&\left\vert Z^a(x,t;\xi,\tau)-Z^{a'}(x,t;\xi,\tau)\right\vert\\
		&\leq \frac{C}{(t-\tau)^\frac{1}{2}}
		\left(\left\vert\frac{1}{a(\xi,\tau)^\frac{1}{2}}-\frac{1}{a'(\xi,\tau)^\frac{1}{2}}\right\vert\right.\\
		&\times\sum_{n=-\infty}^{+\infty}
		\left(\exp\left(-\frac{(x-y_n)^2}{4a(\xi,\tau)(t-\tau)}\right)+\exp\left(-\frac{(x-z_n)^2}{4a(\xi,\tau)(t-\tau)}\right)\right)\\
		&+\frac{1}{a'(\xi,\tau)^\frac{1}{2}}
		\sum_{n=-\infty}^{+\infty}\left\vert
		\exp\left(-\frac{(x-y_n)^2}{4a(\xi,\tau)(t-\tau)}\right)-
		\exp\left(-\frac{(x-y_n)^2}{4a'(\xi,\tau)(t-\tau)}\right)\right\vert\\
		&+\left.\frac{1}{a'(\xi,\tau)^\frac{1}{2}}
		\sum_{n=-\infty}^{+\infty}\left\vert
		\exp\left(-\frac{(x-z_n)^2}{4a(\xi,\tau)(t-\tau)}\right)-
		\exp\left(-\frac{(x-z_n)^2}{4a'(\xi,\tau)(t-\tau)}\right)\right\vert\right),\\
		&\leq \frac{C}{(t-\tau)^\frac{1}{2}}
		\left(
		\frac{\left\vert a(\xi,\tau)-a'(\xi,\tau)\right\vert}{a(\xi,\tau)^\frac{1}{2}a'(\xi,\tau)^\frac{1}{2}(a(\xi,\tau)^\frac{1}{2}+a'(\xi,\tau)^\frac{1}{2})}
		\exp\left(-\frac{(x-\xi)^2}{4\left(a(\xi,\tau)\vee a'(\xi,\tau)\right)(t-\tau)}\right)\right.\\
		&+\left.\frac{1}{a'(\xi,\tau)^\frac{1}{2}}
		\frac{\vert a(\xi,\tau)-a'(\xi,\tau)\vert}{\left(a(\xi,\tau)\wedge a'(\xi,\tau)\right)}
		\exp\left(-\frac{(x-\xi)^2}{4\left(a(\xi,\tau)\vee a'(\xi,\tau)\right)(t-\tau)}\right)\right),\\
		&\leq C\frac{\vert a(\xi,\tau)-a'(\xi,\tau)\vert}{\left(a(\xi,\tau)\wedge a'(\xi,\tau)\right)^{1+\mu}(t-\tau)^\mu\vert x-\xi\vert^{1-2\mu}}
		\exp\left(-\frac{(1-\epsilon)(x-\xi)^2}{4\left(a(\xi,\tau)\vee a'(\xi,\tau)\right)(t-\tau)}\right),
		\end{split}
		\end{equation*}
		where $\mu\in[0,\frac{1}{2}]$. The same method applies to $\frac{\partial Z^a}{\partial x}$ and $\frac{\partial^2 Z^a}{\partial x^2}$.
	\end{proof}
	
	\subsubsection{Regularity properties of Green functions}
	
	In the parametrix method, the fundamental solution (or Green function), is obtained via the following relation (see \cite{Friedman1992}, Chapter 1):
	\begin{equation}
	G^{a,b}(x,t;\xi,\tau):=Z^a(x,t;\xi,\tau)+\int_{\tau}^{t}\int_0^1Z^a(x,t;y,\sigma)\Phi^{a,b}(y,\sigma;\xi,\tau)dyd\sigma,
	\label{Eq:Green_construction}
	\end{equation}
	where $\Phi^{a,b}$ is a perturbation kernel defined in the next proposition. We follow the same procedure as in \cite{Friedman1992}, except that we include the necessary modifications to keep track of coefficients related constants, prove local Lipschitz continuity and use the alternative parametrix from equations \eqref{Eq:Parametrix_Dirichlet} and \eqref{Eq:Parametrix_Neumann}.
	
	\begin{prop}
		Let $(a,b)\in \Lambda_\alpha$ and $Z^a$ be either the function defined in equations \eqref{Eq:Parametrix_Dirichlet} or \eqref{Eq:Parametrix_Neumann}. Consider the following kernel for any $(x,\xi)\in[0,1]^2$ and $0\leq\tau<t\leq T$:
		\begin{equation*}
		(LZ)_1^{a,b}(x,t;\xi,\tau):=\left(a(x,t)-a(\xi,\tau)\right)\frac{\partial^2Z^a}{\partial x^2}+b(x,t)\frac{\partial Z^a}{\partial x},
		\end{equation*}
		and all its iterates for $k\geq 1$:
		\begin{equation}
		(LZ)_{k+1}^{a,b}(x,t;\xi,\tau)=\int_{\tau}^{t}\int_0^1(LZ)_1^{a,b}(x,t;y,\sigma)\left(LZ\right)^{a,b}_k(y,\sigma;\xi,\tau)dyd\sigma.
		\label{Eq:LZk}
		\end{equation}
		Then,
		\begin{equation}
		\Phi^{a,b}(x,t;\xi,\tau):=\sum_{k\geq 1}(LZ)_k^{a,b}(x,t;\xi,\tau),
		\label{Eq:Phi}
		\end{equation}
		is well-defined, 
and	 for $1-\frac{\alpha}{2}<\mu<1$ and $0<\epsilon<1$, is bounded with
			\begin{equation*}
			\begin{split}
			&\left\vert\Phi^{a,b}(x,t;\xi,\tau)\right\vert\\
			&\leq 
			\frac{C}{(t-\tau)^{\mu}\vert x-\xi\vert^{3-2\mu-\alpha}}
			\exp\left(
			-\frac{(1-\epsilon)(x-\xi)^2}{4\lVert a\rVert_{0,\alpha}(t-\tau)}
			+C\left(\frac{\lVert (a,b)\rVert_{0,\alpha}}{m_a^\mu}\right)^{\frac{1}{1-\mu}}
			\right),
			\end{split}
			\end{equation*}
			and  locally Lipschitz continuous in ($a,b$) satisfying
			\begin{equation*}
			\begin{split}
			&\left\vert \Phi^{a,b}(x,t;\xi,\tau)-\Phi^{a',b'}(x,t;\xi,\tau)\right\vert\\
			&\leq C\frac{\lVert (a,b)-(a',b')\rVert_{0,\alpha}}{\left(m_a\wedge m_{a'}\right)^{\frac{5}{2}}(t-\tau)^{\mu}\vert x-\xi\vert^{3-2\mu-\alpha}}\\
			&\times\exp\left(-\frac{(1-\epsilon)(x-\xi)^2}{4\left(\lVert a\rVert_{0,\alpha}\vee\lVert a'\rVert_{0,\alpha}\right)(t-\tau)}+C\left(\frac{\left(\lVert (a,b)\rVert_{0,\alpha}\vee\lVert (a',b')\rVert_{0,\alpha}\right)}{\left(m_a\wedge m_{a'}\right)^\mu}\right)^{\frac{1}{1-\mu}}\right).
			\end{split}
			\end{equation*}
		\label{Prop:Kernel}
	\end{prop}
	
	\begin{proof}
		From the definition of $(LZ)_1^{a,b}$, the bounds given in Proposition \ref{Prop:Parametrix} and because the domain is bounded, it is clear by the triangle inequality (since $m_a\leq 1$, $1\leq \min\left(\lVert a\rVert_{0,\alpha},\lVert b\rVert_{0,\alpha}\right)$) that
		\begin{equation}
		\left\vert (LZ)_1^{a,b}(x,t;\xi,\tau)\right\vert
		\leq C\frac{\lVert (a,b)\rVert_{0,\alpha}}{m_a^{\mu}(t-\tau)^{\mu}\vert x-\xi\vert^{3-2\mu-\alpha}}
		\exp\left(-\frac{(1-\epsilon)(x-\xi)^2}{4a(\xi,\tau)(t-\tau)}\right),
		\label{Eq:Bound_LZ}
		\end{equation}
		where $1-\frac{\alpha}{2}<\mu<1$. Now, using equation \eqref{Eq:Bound_LZ} and \eqref{Eq:LZk}, the kernel $(LZ)_2^{a,b}$ is well-defined when $1-\frac{\alpha}{2}<\mu<1$ and it follows that (using Lemma \ref{Lem:Bound} in the appendix):
		\begin{equation*}
		\begin{split}
		&\left\vert (LZ)_2^{a,b}(x,t;\xi,\tau)\right\vert\\
		&\leq C\frac{\lVert (a,b)\rVert_{0,\alpha}^2}{m_a^{2\mu}(t-\tau)^{\mu+(\mu-1)}\vert x-\xi\vert^{1+2(2-2\mu-\alpha)}}
		\exp\left(-\frac{(1-\epsilon)(x-\xi)^2}{4\lVert a\rVert_{0,\alpha}(t-\tau)}\right).
		\end{split}
		\end{equation*}
		By induction and since $\max(\mu-1,2-2\mu-\alpha)<0$, there exists $k_0\in\mathbb{N}$ such that $\mu+k_0(\mu-1)$ and $1+(k_0+1)(2-2\mu-\alpha)$ are negative, thus can be absorbed in the constant $C$ and it comes:
		\begin{equation*}
		\left\vert (LZ)_{k_0}^{a,b}(x,t;\xi,\tau)\right\vert
		\leq C\left(\frac{\lVert (a,b)\rVert_{0,\alpha}}{m_a^{\mu}}\right)^{k_0}
		\exp\left(-\frac{(1-\epsilon)(x-\xi)^2}{4\lVert a\rVert_{0,\alpha}(t-\tau)}\right).
		\end{equation*}
		From this point onward, one gets $\forall k\geq 0$ (see Friedman \cite{Friedman1992}, page 15):
		\begin{equation*}
		\begin{split}
		&\left\vert (LZ)_{k_0+k}^{a,b}(x,t;\xi,\tau)\right\vert\\
		&\leq C\left(\frac{\lVert (a,b)\rVert_{0,\alpha}}{m_a^{\mu}}\right)^{k_0+m}
		\frac{\left[C(t-\tau)^{1-\mu}\right]^k}{\Gamma\left(k(1-\mu)+1\right)}\exp\left(-\frac{(1-\epsilon)(x-\xi)^2}{4\lVert a\rVert_{0,\alpha}(t-\tau)}\right).
		\end{split}
		\end{equation*}
		Finally, one obtains that $\Phi^{a,b}$ from equation \eqref{Eq:Phi} is well-defined and satisfies the following upper bound (using Lemma \ref{Lem:Bound_MittagLeffler} in the appendix) with $1-\frac{\alpha}{2}<\mu< 1$:
		\begin{equation*}
		\begin{split}
		&\left\vert\Phi^{a,b}(x,t;\xi,\tau)\right\vert\\
		&\leq\frac{C}{(t-\tau)^{\mu}\vert x-\xi\vert^{3-2\mu-\alpha}}\exp\left(-\frac{(1-\epsilon)(x-\xi)^2}{4\lVert a\rVert_{0,\alpha}(t-\tau)}+C\left(\frac{\lVert (a,b)\rVert_{0,\alpha}}{m_a^\mu}\right)^{\frac{1}{1-\mu}}\right).
		\end{split}
		\end{equation*}
		Now, to show the local Lipschitz continuity of $\Phi^{a,b}$, we also proceed by induction. Let $(a,b),(a',b')\in \Lambda_\alpha$, then one has:
		\begin{equation*}
		\begin{split}
		&(LZ)_1^{a,b}(x,t;\xi,\tau)-(LZ)_1^{a',b'}(x,t;\xi,\tau)\\
		&=\left(a(x,t)-a(\xi,\tau)\right)\left(\frac{\partial^2Z^a}{\partial x^2}(x,t;\xi,\tau)-\frac{\partial^2Z^{a'}}{\partial x^2}(x,t;\xi,\tau)\right)\\
		&+\left(a(x,t)-a(\xi,\tau)+a'(\xi,\tau)-a'(x,t)\right)\frac{\partial^2Z^{a'}}{\partial x^2}(x,t;\xi,\tau)\\
		&+b(x,t)\left(\frac{\partial Z^{a}}{\partial x}(x,t;\xi,\tau)-\frac{\partial Z^{a'}}{\partial x}(x,t;\xi,\tau)\right)\\
		&+\left(b(x,t)-b'(x,t)\right)\frac{\partial Z^{a'}}{\partial x}(x,t;\xi,\tau).
		\end{split}
		\end{equation*}
		Now, using Proposition \ref{Prop:Parametrix} one can estimate each term in the right hand side individually:
		\begin{enumerate}
			\item using Proposition \ref{Prop:Parametrix} with $k=2$:
			\begin{equation*}
			\begin{split}
			&\left(a(x,t)-a(\xi,\tau)\right)\left(\frac{\partial^2Z^a}{\partial x^2}(x,t;\xi,\tau)-\frac{\partial^2Z^{a'}}{\partial x^2}(x,t;\xi,\tau)\right)\\
			&\leq \frac{C\lVert a\rVert_\infty}{\left(m_a\wedge m_{a'}\right)^{\frac{5}{2}+\mu}(t-\tau)^\mu\vert x-\xi\vert^{2-2\mu}}\\
			&\times\exp\left(-\frac{(1-\epsilon)(x-\xi)^2}{4\left(\lVert a\rVert_{0,\alpha}\vee\lVert a'\rVert_{0,\alpha}\right)(t-\tau)}\right)\lVert a-a'\rVert_{0,\alpha},
			\end{split}
			\end{equation*}
			\item as we have $a-a'\in\mathcal{C}^{0,\alpha}([0,1]\times[0,T],\mathbb{R})$ and using Proposition \ref{Prop:Parametrix} with $k=2$:
			\begin{equation*}
				\begin{split}
				&\left(a(x,t)-a(\xi,\tau)+a'(\xi,\tau)-a'(x,t)\right)\frac{\partial^2Z^{a'}}{\partial x^2}(x,t;\xi,\tau)\\
				&\leq \frac{C\lVert a-a'\rVert_{0,\alpha}}{a'(\xi,\tau)^\mu(t-\tau)^\mu\vert x-\xi\vert^{3-2\mu-\alpha}}\exp\left(-\frac{(1-\epsilon)(x-\xi)^2}{4a'(\xi,\tau)(t-\tau)}\right)
				\end{split}
			\end{equation*}
			\item here Proposition \ref{Prop:Parametrix} with $k=1$ gives:
			\begin{equation*}
				\begin{split}
				&b(x,t)\left(\frac{\partial Z^{a}}{\partial x}(x,t;\xi,\tau)-\frac{\partial Z^{a'}}{\partial x}(x,t;\xi,\tau)\right)\\
				&\leq \frac{C\lVert b\rVert_\infty}{\left(m_a\wedge m_{a'}\right)^{\frac{3}{2}+\mu}(t-\tau)^\mu\vert x-\xi\vert^{1-2\mu}}\\
				&\times\exp\left(-\frac{(1-\epsilon)(x-\xi)^2}{4\left(\lVert a\rVert_{0,\alpha}\vee\lVert a'\rVert_{0,\alpha}\right)(t-\tau)}\right)\lVert a-a'\rVert_{0,\alpha},
				\end{split}
			\end{equation*}
			\item and finally:
			\begin{equation*}
			\begin{split}
			&\left(b(x,t)-b'(x,t)\right)\frac{\partial Z^{a'}}{\partial x}(x,t;\xi,\tau)\\
			&\leq\frac{C\lVert b-b'\rVert_{0,\alpha}}{a'(\xi,\tau)^\mu(t-\tau)^\mu\vert x-\xi\vert^{2-2\mu}}\exp\left(-\frac{(1-\epsilon)(x-\xi)^2}{4a(\xi,\tau)(t-\tau)}\right), 
			\end{split}
			\end{equation*}
		\end{enumerate}
		and gets $1-\frac{\alpha}{2}<\mu<1$:
		\begin{equation*}
		\begin{split}
		&\left\vert(LZ)_1^{a,b}(x,t;\xi,\tau)-(LZ)_1^{a',b'}(x,t;\xi,\tau)\right\vert\\
		&\leq C\frac{\left(\lVert (a,b)\rVert_{0,\alpha}\vee\lVert (a',b')\rVert_{0,\alpha}\right)}{\left(m_a\wedge m_{a'}\right)^{\frac{5}{2}+\mu}(t-\tau)^\mu\vert x-\xi\vert^{3-2\mu-\alpha}}\\
		&\times\exp\left(-\frac{(1-\epsilon)(x-\xi)^2}{4\left(\lVert a\rVert_{0,\alpha}\vee\lVert a'\rVert_{0,\alpha}\right)(t-\tau)}\right)\lVert (a,b)-(a',b')\rVert_{0,\alpha}.
		\end{split}
		\end{equation*}
		For the first iterated kernel, one has:
		\begin{equation*}
		\begin{split}
		&\left\vert LZ_2^{a,b}(x,t;\xi,\tau)-LZ_2^{a',b'}(x,t;\xi,\tau)\right\vert\\
		&\leq \int_{\tau}^{t}\int_0^1 \left\vert LZ^{a,b}_1(x,t;y,\sigma)\right\vert
		\left\vert LZ_1^{a,b}(y,\sigma;\xi,\tau)-LZ_1^{a',b'}(y,\sigma;\xi,\tau)\right\vert dyd\sigma\\
		&+ \int_{\tau}^{t}\int_0^1
		\left\vert LZ_1^{a,b}(x,t;y,\sigma)-LZ_1^{a',b'}(x,t;y,\sigma)\right\vert
		\left\vert LZ^{a',b'}_1(y,\sigma;\xi,\tau\right\vert dyd\sigma\\
		&\leq C\frac{\left(\lVert (a,b)\rVert_{0,\alpha}\vee\lVert (a',b')\rVert_{0,\alpha}\right)^2\lVert (a,b)-(a',b')\rVert_{0,\alpha}}{\left(m_a\wedge m_{a'}\right)^{\frac{5}{2}+2\mu}(t-\tau)^{\mu+(\mu-1)}\vert x-\xi\vert^{1+2(2-2\mu-\alpha)}}\\
		&\times\exp\left(-\frac{(1-\epsilon)(x-\xi)^2}{4\left(\lVert a\rVert_{0,\alpha}\vee\lVert a'\rVert_{0,\alpha}\right)(t-\tau)}\right),
		\end{split}
		\end{equation*}
		and just as before, there exists $k_0\in\mathbb{N}^*$ such that
		\begin{equation*}
		\begin{split}
		&\left\vert (LZ)_{k_0}^{a,b}(x,t;\xi,\tau)-(LZ)_{k_0}^{a',b'}(x,t;\xi,\tau)\right\vert\\
		&\leq C\frac{\left(\lVert (a,b)\rVert_{0,\alpha}\vee\lVert (a',b')\rVert_{0,\alpha}\right)^{k_0}}{\left(m_a\wedge m_{a'}\right)^{\frac{5}{2}+k_0\mu}}\\
		&\exp\left(-\frac{(1-\epsilon)(x-\xi)^2}{4\left(\lVert a\rVert_{0,\alpha}\vee\lVert a'\rVert_{0,\alpha}\right)(t-\tau)}\right)\lVert (a,b)-(a',b')\rVert_{0,\alpha}.
		\end{split}
		\end{equation*}
		Now, using Lemma \ref{Lem:Bound_MittagLeffler} as before, the function $\Phi^{a,b}$ is locally Lispchitz continuous with the following constant:
		\begin{equation*}
		\begin{split}
		&\left\vert\Phi^{a,b}(x,t;\xi,\tau)-\Phi^{a',b'}(x,t;\xi,\tau)\right\vert\\
		&\leq \frac{C\lVert (a,b)-(a',b')\rVert_{0,\alpha}}{\left(m_a\wedge m_{a'}\right)^{\frac{5}{2}}(t-\tau)^{\mu}\vert x-\xi\vert^{3-2\mu-\alpha}}\\
		&\exp\left(-\frac{(1-\epsilon)(x-\xi)^2}{4\left(\lVert a\rVert_{0,\alpha}\vee\lVert a'\rVert_{0,\alpha}\right)(t-\tau)}+C\left(\frac{\left(\lVert (a,b)\rVert_{0,\alpha}\vee\lVert (a',b')\rVert_{0,\alpha}\right)}{\left(m_a\wedge m_{a'}\right)^\mu}\right)^{\frac{1}{1-\mu}}\right),
		\end{split}
		\end{equation*}
		where $1-\frac{\alpha}{2}<\mu<1$ and $0<\epsilon<1$.
	\end{proof}
	
	Now that the necessary properties have been derived for the perturbation kernel $\Phi^{a,b}$, it remains to show that they are carried over to Green functions using equation \eqref{Eq:Green_construction}. This will finish the proof of Theorem \ref{Thm:ParabolicGreen}.
	
	\begin{proof}[Proof of Theorem \ref{Thm:ParabolicGreen}]
		Let $G^{a,b}$ defined as in equation \eqref{Eq:Green_construction} with $Z^a$ being either $Z^a_D$ from equation \eqref{Eq:Parametrix_Dirichlet} or $Z^a_N$ from equation \eqref{Eq:Parametrix_Neumann}. The fact that it is indeed a fundamental solution is proven in Theorem 8, Chapter 1 from \cite{Friedman1992}. In both cases ($Z^a_N$ or $Z^a_D$), the boundary conditions are obtained immediately. Now, applying results from Propositions \ref{Prop:Parametrix} and \ref{Prop:Kernel}, one has for all $(a,b)\in \Lambda_\alpha$:
		\begin{equation*}
		\begin{split}
		&\left\vert G^{a,b}(x,t;\xi,\tau)\right\vert\\
		&\leq \frac{C}{a(\xi,\tau)^{\mu}(t-\tau)^{\mu}\vert x-\xi\vert^{1-2\mu}}\exp\left(-\frac{(1-\epsilon)(x-\xi)^2}{4\lVert a\rVert_{0,\alpha}(t-\tau)}\right)\\
		&+C\frac{\exp\left(C\left(\frac{\lVert (a,b)\rVert_{0,\alpha}}{m_a^{\mu^*}}\right)^{\frac{1}{1-\mu^*}}-\frac{(1-\epsilon)(x-\xi)^2}{4\lVert a\rVert_{0,\alpha}(t-\tau)}\right)}{m_a^{\mu}(t-\tau)^{\mu+(\mu^*-1)}\vert x-\xi\vert^{1-2\mu+(2-2\mu^*-\alpha)}},
		\end{split}
		\end{equation*}
		where $0\leq \mu\leq \frac{1}{2}$ and $1-\frac{\alpha}{2}<\mu^*<1$. Since $\max(\mu^*-1,2-2\mu^*-\alpha)<0$ it becomes $\forall q>\frac{2}{\alpha}$:
		\begin{equation*}
		\begin{split}
		&\left\vert G^{a,b}(x,t;\xi,\tau)\right\vert\\
		&\leq\frac{C}{m_a^{\mu}(t-\tau)^{\mu}\vert x-\xi\vert^{1-2\mu}}\exp\left(-\frac{(1-\epsilon)(x-\xi)^2}{4\lVert a\rVert_{0,\alpha}(t-\tau)}+Cm_a^{1-q}\lVert (a,b)\rVert_{0,\alpha}^q\right),
		\end{split}
		\end{equation*}
		with $0\leq\mu\leq\frac{1}{2}$. In particular, the choice of $\mu=\frac{1}{2}$ gives:
		\begin{equation*}
			\left\vert G^{a,b}(x,t;\xi,\tau)\right\vert
			\leq\frac{C}{(t-\tau)^{\frac{1}{2}}}\exp\left(Cm_a^{1-q}\lVert (a,b)\rVert_{0,\alpha}^q\right).
		\end{equation*}
		The local Lipschitz continuity is obtained with the same method, simply rewriting the expression as follows and using Propositions \ref{Prop:Parametrix} and \ref{Prop:Kernel}:
		\begin{equation*}
		\begin{split}
		&\left\vert G^{a,b}(x,t;\xi,\tau)-G^{a',b'}(x,t;\xi,\tau)\right\vert\\
		&\leq \left\vert Z^a(x,t;\xi,\tau)-Z^{a'}(x,t;\xi,\tau)\right\vert\\
		&+\int_{\tau}^{t}\int_0^1Z^a(x,t;y,\sigma)\left\vert\Phi^{a,b}(y,\sigma;\xi,\tau)-\Phi^{a',b'}(y,\sigma;\xi,\tau)\right\vert dyd\sigma\\
		&+\int_{\tau}^{t}\int_0^1\left\vert Z^{a}(x,t;y,\sigma)-Z^{a'}(x,t;y,\sigma)\right\vert\left\vert\Phi^{a',b'}(y,\sigma;\xi,\tau)\right\vert dyd\sigma,
		\end{split}
		\end{equation*}
		and the same argument applies.
	\end{proof}
	
	\subsection{Proofs of the well-posedness of Bayesian inference of the coefficients}
	\label{subsec:SDE}
	
	\begin{proof}[Proof of Theorem \ref{Thm:Main_thm}]
		By Theorem \ref{Thm:ParabolicGreen}, the likelihood is continuous in $(a,b,y,s)$. 
		Therefore to establish that $\mu^{y,s}$ is well-defined, it remains to show that $0<\cZ(y,s)<\infty$.
		Setting the constant constant $b^0(x,t)=0$ and $a^0(x,t)=1$ for $x\in[0,1]$ and $t\in(0,T)$, one has that $\mathcal{L}^{y,s}(a^0,b^0)>0$ (as $p^{a^0,b^0}(x,t;\xi,\tau)=Z^{a^0}(x,t,\xi,\tau)$ from equations \eqref{Eq:Parametrix_Dirichlet} or \eqref{Eq:Parametrix_Neumann} and $y\in(0,1)^n$). Since the likelihood is locally Lipschitz in $(a,b)$, there exists a measurable set $A\subset \Lambda_\alpha$ containing $(a^0,b^0)\in\Lambda_\alpha$ with $\mu_0(A)>0$ and such that $\mathcal{L}^{y,s}(a,b)>0$ for all $(a,b)\in A$. 
		Hence
		$$
		\cZ(y,s)\ge\int_A\mathcal{L}^{y,s}(a,b)\,d\mu_0>0.
		$$
		Now, using the upper bound on the transition probability density function from Theorem \ref{Thm:ParabolicGreen}, it foloows that there exists $C>0$ such that:
		\begin{equation*}
		\forall q>\frac{2}{\alpha},\;\mathcal{L}^{a,b}\left(y,s\right)=\prod_{i=1}^{n-1}p^{(a,b)}(y_{i+1},s_{i+1};y_i,s_i)\leq C\exp\left(Cm_a^{1-q}\lVert (a,b)\rVert^q_{0,\alpha}\right),
		\label{Eq:BoundLike}
		\end{equation*}
		and hence $\cZ(y,s)<+\infty$ by the hypothesis on $\mu_0$. The posterior measure is then uniquely defined by the density in equation \eqref{Eq:RadonNikodym}. The last step is to show the continuity in Hellinger's metric. This follows along the lines of the proof of \cite[Lemma 3.7]{Latz2020} since by Theorem \ref{Thm:ParabolicGreen} $\mathcal{L}^{y,s}(a,b)$ is continuous in $y$ and $s$ and for any fixed $(y,s)$ is bounded by a $\mu_0$-integrable function of $(a,b)$.
	\end{proof}

	\begin{proof}[Proof of Theorem \ref{Thm:Approximation}]
		By the hypothesis, one has $(a,b), (a_k,b_k)\in \Lambda_\alpha$ with $\alpha\in(0,1)$ for all $k\in\mathbb{N}$ and thus Theorem $\ref{Thm:Main_thm}$ implies that $\mu^{a,b}$ and $\mu_k^{a,b}$, for any $k\in\mathbb{N}$, are well defined. Now, observe that we have:
		\begin{equation*}
		\begin{split}
		&\left\vert \mathcal{L}^{y,s}(a,b)-\mathcal{L}^{y,s}\left(a_k,b_k\right)\right\vert\\
		&\leq \left\vert\prod_{i=1}^{n-1}p^{a,b}(x_{i+1},s_{i+1},x_i,s_i)-\prod_{i=1}^{n-1}p^{a_k,b_k}(x_{i+1},s_{i+1},x_i,s_i)\right\vert,\\
		&\leq C\exp\left(C\left(m_a\wedge m_{a_k}\right)^{1-q}\left(\lVert (a,b)\rVert_{0,\alpha}\vee\lVert (a_k,b_k)\rVert_{0,\alpha}\right)^q\right)\lVert (a,b)-(a_k,b_k)\rVert_{0,\alpha},\\
		&\leq C\exp\left(Cm_a^{1-q}\lVert (a,b)\rVert_{0,\alpha}^q\right)\psi(k),
		\end{split}
		\end{equation*}
		from which it follows that for all $k\in\mathbb{N}$:
		\begin{equation*}
			\left\vert \cZ(y,s)-\cZ_k(y,s)\right\vert\leq C\psi(k).
		\end{equation*}
		In particular, for all $\epsilon>0$ there exists $N\in\mathbb{N}$ such that for all $k\geq N$, $\left\vert \cZ_k(y,s)-\cZ(y,s)\right\vert\leq\epsilon$ and since $\cZ(y,s)>0$, one can find a positive constant $c$ such that $\cZ(y,s)>c$ and $\cZ_k(y,s)>c$ for all $n$ large enough. Now, following \cite{Stuart2010}, we have
		\begin{equation*}
		\begin{split}
		d_{\rm H}\left(\mu^y,\mu^y_k\right)^2&=\frac{1}{2}\int_{\Lambda_\alpha}\left(\sqrt{\frac{\mathcal{L}^{y,s}(a,b)}{\cZ(y,s)}}-\sqrt{\frac{\mathcal{L}^{y,s}\left(a_k,b_k\right)}{\cZ_k(y,s)}}\right)^2d\mu_0,\\
		&\leq \frac{1}{2}\int_{\Lambda_\alpha}\left\vert\frac{\mathcal{L}^{y,s}(a,b)}{\cZ(y,s)}-\frac{\mathcal{L}^{y,s}\left(a_k,b_k\right)}{\cZ_k(y,s)}\right\vert d\mu_0,\\
		&\leq I_1+I_2,
		\end{split}
		\end{equation*}
		where we used $\left(\sqrt{x}-\sqrt{y}\right)^2\leq \vert x-y\vert$, for all $(x,y)\in(0,\infty)^2$ and:
		\begin{equation*}
			\begin{split}
			I_1&=\cZ(y,s)^{-1}\int_{\Lambda_\alpha}\left\vert\mathcal{L}^{y,s}(a,b)-\mathcal{L}^{y,s}\left(a_k,b_k\right)\right\vert d\mu_0,\\
			I_2&=\left\vert \cZ(y,s)^{-1}-\cZ_k(y,s)^{-1}\right\vert \cZ_k(y,s).
			\end{split}
		\end{equation*}
		This gives that $d_{\rm H}(\mu^{y,s},\mu_k^{y,s})\leq C\sqrt{\psi(k)}$ for all $k$ sufficiently large.
	\end{proof}
	
	\subsection{Application to the BD process with sub-Gaussian exponential priors}\label{s:BDproofs}
	
	\begin{proof}[Proof of Proposition \ref{Prop:RandomSeries}]
		Let $\beta>2$, $l> \alpha+\frac{1}{2}$, $(f_k)$ be the Fourier basis in $L^2(0,1)$. Our objective is to show that $f=\sum_{k\geq 0}\gamma_k\eta_kf_k$ has an exponential moment of order $\beta$ in the $H^l$ norm, that is:
		\begin{equation*}
			\mathbb{E}\left[\exp\left(C\lVert f\rVert_{H^l}^\beta\right)\right]<\infty.
		\end{equation*}
		First of all, observe that:
		\begin{equation*}
		\lVert f\rVert_{H^l}^2=\sum_{k\geq 0}k^{2l}\gamma_k^2\eta_k^2\leq c_2\sum_{k\geq 0}k^{2(l-\theta)}\eta_k^2.
		\end{equation*}
		A direct application of H\"older inequality gives:
		\begin{equation*}
		\begin{split}
		\lVert f\rVert_{H^l}^\beta
		&\leq C\left(\sum_{k\geq 0}k^{2(l-\theta)}\eta_k^2\right)^\frac{\beta}{2}
		\leq
		C\left(\sum_{k\geq 0}k^{q(2l-\theta)}\right)^\frac{\beta}{2q}
		\left(\sum_{k\geq 0}k^{-\theta p}\left\vert\eta_k\right\vert^{2p}\right)^\frac{\beta}{2p},
		\end{split}
		\end{equation*}
		where $\frac{1}{p}+\frac{1}{q}=1$. Now, since $\beta>2$ choosing $p=\frac{\beta}{2}$ implies $q=\frac{\beta}{\beta-2}$ and
		\begin{equation}
		\lVert f\rVert_{H^l}^\beta
		\leq
		C\left(\sum_{k\geq 0}k^{\frac{\beta(2l-\theta)}{\beta-2}}\right)^\frac{\beta-2}{2}
		\sum_{k\geq 0}k^{-\frac{\theta\beta}{2}}\left\vert\eta_k\right\vert^\beta.
		\label{Eq:HolderBeta}
		\end{equation}
		Suppose now that $\theta>2l+1-2\beta^{-1}$, then the first term in the right hand-side of equation \eqref{Eq:HolderBeta} is a constant $C$. We deduce that, following \cite{Lassas2009}:
		\begin{equation}
		\begin{split}
		\mathbb{E}\left[\exp\left(\lVert f\rVert^\beta_{H^l}\right)\right]
		&\leq\mathbb{E}\left[\exp\left(C\sum_{k\geq 0}k^{-\frac{\theta\beta}{2}}\eta_k^\beta\right)\right],\\
		&\leq \prod_{k\geq 0}\mathbb{E}\left[\exp\left(Ck^{-\frac{\theta\beta}{2}}\eta_k^\beta\right)\right],\\
		&\leq \prod_{k\geq 0}\left(1-2Ck^{-\frac{\theta\beta}{2}}\right)^{-\frac{1}{\beta}},
		\end{split}
		\label{Eq:ExponentialMoment}
		\end{equation}
		where we used Fubini and the independence of the sequence $(\eta_k)$. The right-hand-side of equation \eqref{Eq:ExponentialMoment} converges if $2C<1$ and $\theta>\frac{2}{\beta}$. Finally, choosing $\theta>2l+1-2\beta^{-1}$ gives the required exponential moment of order $\beta>2$. Now, using the Sobolev embedding Theorem (see for instance \cite{Evans1998}), one has that $H^l$ injects continuously in $\mathcal{C}^{0,\alpha}$ for all $l> \alpha+\frac{1}{2}$. The almost-sure convergence of $f$ in $H^l$ follows immediately from the finiteness of its moments.
	\end{proof}
	
	\begin{proof}[Proof of Proposition \ref{Prop:WellposedU}]
		Let $l>\frac{1}{2}+\alpha$ with $\alpha\in(0,1]$ then clearly $U=g(f)h\in H^l$ and we have $\lVert U\rVert_{H^l}\leq C\lVert f\rVert_{H^l}+C$ with some $C>0$. Using the relation in equation \eqref{Eq:abU}, the distribution of $U$ defines a measure $\mu_0$ on $\Lambda_\alpha$ compatible with Theorem \ref{Thm:Main_thm}. Indeed, it is clear that $(a,b)\in \Lambda_\alpha$ almost-surely (since $U(0)>0$). Now, let us show that this gives the required integrability condition. We have 
		\begin{equation*}
		\begin{split}
		m_a=\min_{x\in[0,1]}a(x)&=\frac{1}{N}\min_{x\in[0,1]}\left(g(f(x))\left(1-\exp(x-1)\right)+\gamma x\right),\\
		&\geq\frac{1}{N}\min_{x\in[0,1]}\left(g(-\lVert f\rVert_\infty)\left(1-\exp(x-1)\right)+\gamma x\right),\\
		&\geq\frac{1}{N}\min_{x\in[0,1]}\left(\frac{1-\exp(x-1)}{2+\lVert f\rVert_\infty}+\gamma x\right),\\
		&\geq\frac{1}{N}\min\left(\frac{1-e^{-1}}{2+\lVert f\rVert_\infty},\gamma\right),\\
		&\geq\frac{\min(1-e^{-1},\gamma)}{N(2+\lVert f\rVert_\infty)},\\
		&\geq C\left(1+\lVert f\rVert_{H^l}\right)^{-1}>0,
		\end{split}
		\end{equation*}
		and hence, for any $q>\frac{2}{\alpha}$,
		\begin{equation*}
		\begin{split}
		&\mathbb{E}^{\mu_0}\left[\exp\left(Cm_a^{1-q}\lVert (a,b)\rVert_{0,\alpha}^q\right)\right],\\
		&\leq\mathbb{E}^{\mu_0}\left[\exp\left(C\left(1+\lVert f\rVert_{H^l}\right)^{q-1}\left(C\lVert f\rVert_{H^l}+C\right)^q\right)\right],\\
		&\leq\mathbb{E}^{\mu_0}\left[C\exp\left(C\lVert f\rVert_{H^l}^{2q-1}\right)\right].
		\end{split}
		\end{equation*}
		It is then enough, by Theorem \ref{Thm:Main_thm}, to choose $\beta>2q-1$ and finally $\theta>2l+1-2\beta^{-1}$ to obtain a well-defined posterior measure $\mu^{y,s}$ on $\Lambda_{\alpha}$. The result then follows after observing that for the $\mu_0$ considered in this proof we have $\cZ^\nu(y,s)=\cZ(y,s)$, and for any given $A\in\mathcal{B}(C_0^{0,\alpha})$, setting
		$$
		B=\Big\{(\frac{U+D}{N}, U-D)\in\Lambda_{\alpha}:U\in A\Big\},
		$$
		it holds that $\nu^{y,s}(A)=\mu^{y,s}(B)$.
	\end{proof}

\begin{proof}[Proof of Proposition \ref{p:MAPs}]

The probability measure $\nu_0$ is convex \cite[Section 4.3.3]{Bogachev2010} and differentiable along any $w\in H^{\theta}$ (see Section 5.3.2 of \cite{Bogachev2010}). For any $w=\sum_{k\ge 0}w_k f_k\in H^{\theta}$, the logarithmic derivative of $\mu_0$ along $w$, $\iota_w$, can be written as
\begin{align*}
\iota_w(U)=\lim_{K\to\infty}\sum_{k=1}^{K}\iota_{w_k}(u_k),~~\mbox{ in } L^1
\end{align*}
with
\begin{align*}
\iota_{w_k}(u_k)=u_k\left(-\beta{\rm sign}(u_k)\gamma_k^{-\beta}|u_k|^{\beta-1}\right)
\end{align*}
for any $U=\sum_{k\ge 1}u_k f_k\in C^{0,\alpha}_0$. This follows by arguing along the lines of the proof of Theorem 6 of \cite{Helin2015}. We then can conclude from \cite[Corollary 2]{Helin2015} and \cite{Lie2018a} that the MAP estimators of posterior measure for $U|y$ are given by minimisers of
\begin{equation*}
 -\log \mathcal{L}^{y,s}_{BD}(U)+\|g^{-1}(U/h)\|_{E}^\beta.
\end{equation*}

\end{proof}

	\begin{proof}[Proof of Proposition \ref{Prop:ApproxU}]
		The proof of Proposition \ref{Prop:WellposedU} is readily adapted when $U$ is replaced by $U_k$ and leads to a well-defined posterior measure $\mu_k^{y,s}$, continuous in the data $y$ with respect to Hellinger's metric for all $k\in\mathbb{N}$. Now, regarding the approximation, for any $\alpha\in(0,1]$ and $l>\frac{1}{2}+\alpha$, we have
		\begin{equation*}
		\lVert (a,b)-(a_k,b_k)\rVert_{0,\alpha}\leq C\lVert (a,b)-(a_k,b_k)\rVert_{H^l}\leq C\lVert U-U_k\rVert_{H^l}\leq C\lVert f-P_kf\rVert_{H^l}.
		\end{equation*}
		Now, observe that, as was done in \cite{Cotter2010} for any $l> l^*>\alpha+\frac{1}{2}$:
		\begin{equation*}
		\begin{split}
		\lVert f-P_kf\rVert_{H^{l^*}}
		&\leq C\left(\sum_{i\geq k+1}i^{2(l^*-\theta)}\eta_i^2\right)^\frac{1}{2},\\
		&\leq C\left(k+1\right)^{l^*-l}\left(\sum_{i\geq k+1}i^{2(l-\theta)}\eta_i^2\right)^\frac{1}{2},\\
		&\leq C\left(k+1\right)^{l^*-l}\lVert f\rVert_{H^l},
		\end{split}
		\end{equation*}
		so that finally for $l>l^*>\alpha+\frac{1}{2}$, $\lVert (a,b)-(a_k,b_k)\rVert_{0,\alpha}\leq \lVert f\rVert_{H^l}(k+1)^{l^*-l}$. Now similarly to the proof of Proposition \ref{Prop:WellposedU}, we have that for all $k\in\mathbb{N}$ and $l> \alpha+\frac{1}{2}$:
		\begin{equation*}
			m_{a_k}\geq C\left(1+\lVert P_kf\rVert_{H^l}\right)^{-1}\geq C\left(1+\lVert f\rVert_{H^l}\right)^{-1}.
		\end{equation*}
		The rest of the proof is similar to Theorem \ref{Thm:Approximation}, starting from the inequality:
		\begin{equation*}
			\left\vert \mathcal{L}^{y,s}(a,b)-\mathcal{L}^{y,s}\left(a_k,b_k\right)\right\vert
			\leq C\exp\left(C\lVert f\rVert_{H^l}^{2q-1}\right)(k+1)^{l^*-l}.
		\end{equation*}
		Taking finally $l^*\to \alpha+\frac{1}{2}$ gives the announced result.
	\end{proof}

	\appendix
	
	\section{Auxiliary Lemmas}
	
	\begin{lem}
		Let $T>0$, $(\alpha_1,\alpha_2,\alpha_3,\alpha_4)\in(0,1)^4$, $h>0$, $0\leq\tau<t\leq T$ and $(x,\xi)\in[0,1]^2$ then there exists $C>0$ such that:
		\begin{equation}
		\begin{split}
		&\int_{\tau}^{t}\int_{0}^1\frac{1}{(t-\sigma)^{\alpha_1}(\sigma-\tau)^{\alpha_2}\vert x-y\vert^{\alpha_3}\vert y-\xi\vert^{\alpha_4}}\exp\left(-\frac{h( x-y)^2}{4(t-\sigma)}-\frac{h( y-\xi)^2}{4(\sigma-\tau)}\right)dyd\sigma\\
		&\leq \frac{C}{(t-\tau)^{\alpha_1+\alpha_2-1}\vert x-\xi\vert^{\alpha_3+\alpha_4-1}}\exp\left(-\frac{h(x-\xi)^2}{4(t-\tau)}\right).
		\label{Eq:Lem_Bound}
		\end{split}
		\end{equation}
		\label{Lem:Bound}
	\end{lem}

	\begin{proof}
		Let $D$ be the left hand side of equation \eqref{Eq:Lem_Bound} and start by completing the square, observing that:
		\begin{equation*}
			z^2+\frac{h(x-\xi)^2}{4(t-\tau)}=\frac{h(x-y)^2}{4(t-\sigma)}+\frac{h( y-\xi)^2}{4(\sigma-\tau)}.
		\end{equation*}
		It immediately follows that:
		\begin{equation}
			D\leq\exp\left(-\frac{h(x-\xi)^2}{4(t-\tau)}\right)
			\int_{\tau}^{t}\frac{1}{(t-\sigma)^{\alpha_1}(\sigma-\tau)^{\alpha_2}}d\sigma
			\int_{0}^1\frac{1}{\vert x-y\vert^{\alpha_3}\vert y-\xi\vert^{\alpha_4}}dy.
			\label{Eq:Lem_Bound2}
		\end{equation}
		Now let $s=\frac{\sigma-\tau}{t-\tau}$, then it comes that:
		\begin{equation*}
			\int_{\tau}^{t}\frac{1}{(t-\sigma)^{\alpha_1}(\sigma-\tau)^{\alpha_2}}d\sigma=\frac{B(1-\alpha_1,1-\alpha_2)}{(t-\tau)^{\alpha_1+\alpha_2-1}},
		\end{equation*}
		where $B$ is the Beta function. Using a similar argument for the second integral in the right hand side of equation \eqref{Eq:Lem_Bound2} ends the proof.
	\end{proof}

	\begin{lem}
		Let $(a,0),(a',0)\in \Lambda_\alpha$. Then for all $(\xi,\tau)\in [0,1]\times[0,T]$ and $\beta\geq 1$:
		\begin{equation*}
		\left\vert\frac{1}{a(\xi,\tau)^\beta}-\frac{1}{a'(\xi,\tau)^\beta}\right\vert
		\leq\frac{\beta\vert a(\xi,\tau)-a'(\xi,\tau)\vert}{a'(\xi,\tau)a(\xi,\tau)\left(a(\xi,\tau)\wedge a'(\xi,\tau)\right)^{\beta-1}}.
		\end{equation*}
		The following inequality also holds for all $(x,t;\xi,\tau)$, $0\leq \tau<t$:
		\begin{equation*}
		\begin{split}
		&\left\vert\exp\left(-\frac{(x-\xi)^2}{4a(\xi,\tau)(t-\tau)}\right)-\exp\left(-\frac{(x-\xi)^2}{4a'(\xi,\tau)(t-\tau)}\right)\right\vert\\
		&\leq C\exp\left(-\frac{(1-\epsilon)(x-\xi)^2}{4\left(a(\xi,\tau)\vee a'(\xi,\tau)\right)(t-\tau)}\right)\frac{\vert a(\xi,\tau)-a'(\xi,\tau)\vert}{\left(a(\xi,\tau)\wedge a'(\xi,\tau)\right)}
		\end{split}
		\end{equation*}
		\label{Lem:Taylor}
	\end{lem}

	\begin{proof}
		The proof directly follows from the mean value Theorem. Indeed it comes that:
		\begin{equation*}
			\forall x>0,\;x^\beta=\left(x-1+1\right)^\beta=1+\vert x-1\vert\beta c^{\beta-1},\;c\in\left(\min\left(1, x\right),\max\left(1, x\right)\right).
		\end{equation*}
		Applying this relation with $x=\frac{a'(\xi,\tau)}{a(\xi,\tau)}$ gives:
		\begin{equation*}
		\left\vert\frac{1}{a(\xi,\tau)^\beta}-\frac{1}{a'(\xi,\tau)^\beta}\right\vert
		=\frac{1}{a'(\xi,\tau)^\beta}\left\vert \left(\frac{a'(\xi,\tau)}{a(\xi,\tau)}\right)^\beta-1\right\vert
		\leq \frac{\beta\vert a(\xi,\tau)-a'(\xi,\tau)\vert c^{\beta-1}}{a'(\xi,\tau)^\beta}.
		\end{equation*}
		Here $a(\xi,\tau)c\in\left(\min\left(a(\xi,\tau), a'(\xi,\tau)\right),\max\left(a(\xi,\tau), a'(\xi,\tau)\right)\right)$,
		\begin{equation*}
		\left\vert\frac{1}{a(\xi,\tau)^\beta}-\frac{1}{a'(\xi,\tau)^\beta}\right\vert
		\leq\frac{\beta\vert a(\xi,\tau)-a'(\xi,\tau)\vert }{a(\xi,\tau)a'(\xi,\tau)\left(a(\xi,\tau)\wedge a'(\xi,\tau)\right)^{\beta-1}}.
		\end{equation*}
		Now, using the inequality $1-\exp(-x)\leq x,\;\forall x\geq 0$, one has:
		\begin{equation*}
		\begin{split}
		&\left\vert
		\exp\left(-\frac{(x-\xi)^2}{4a(\xi,\tau)(t-\tau)}\right)
		-\exp\left(-\frac{(x-\xi)^2}{4a'(\xi,\tau)(t-\tau)}\right)
		\right\vert\\
		&\leq\exp\left(-\frac{(x-\xi)^2}{4\left(a(\xi,\tau)\vee a'(\xi,\tau)\right)}\right)
		\frac{(x-\xi)^2}{4(t-\tau)}
		\frac{\vert a(\xi,\tau)-a'(\xi,\tau)\vert}{a(\xi,\tau)a'(\xi,\tau)},\\
		&\leq\exp\left(-\frac{(1-\epsilon)(x-\xi)^2}{4\left(a(\xi,\tau)\vee a'(\xi,\tau)\right)}\right)
		\frac{\vert a(\xi,\tau)-a'(\xi,\tau)\vert}{\left(a(\xi,\tau)\wedge a'(\xi,\tau)\right)}.
		\end{split}
		\end{equation*}
	\end{proof}

	\begin{lem}
		Let $E_\alpha(x)$ be the one-parameter Mittag-Leffler function, defined as follows:
		\begin{equation*}
		E_\alpha(x)=\sum_{k\geq 0}\frac{x^k}{\Gamma(\alpha k+1)},
		\end{equation*}
		then there is $C>0$ such that
		\begin{equation*}
		E_\alpha(x)\leq C\exp\left( Cx^{\frac{1}{\alpha}}\right).
		\end{equation*}
		\label{Lem:Bound_MittagLeffler}
	\end{lem}
	
	This last Lemma can be proved using a standard series/integral comparison argument.

\section*{Acknowledgments} The authors acknowledge support from the Leverhulme Trust for the Research Project Grant RPG2017-370.

	\bibliographystyle{abbrv}
	\bibliography{Biblio}
\end{document}